\documentclass[11pt,a4paper]{article}

\usepackage[top=2.5cm, bottom = 2.5cm, left = 2.5cm, right = 2.5cm]{geometry}

\usepackage{package_RF}
\usepackage{graphbox}

\newcommand{\prBm}{partially reflected Brownian motion\xspace}

\title{Gene flow across geographical barriers - scaling limits of random walks with obstacles}
\author{Rapha\"el Forien\footnote{INRA - Unité de Biostatistiques et Processus Spatiaux (UR546), Domaine St-Paul - Site Agroparc, 84914 Avignon Cedex, FRANCE\newline raphael.forien@inra.fr}}

\begin{document}
	
	\maketitle

	\begin{abstract}
		We study a class of random walks which behave like simple random walks outside of a bounded region around the origin and which are subject to a partial reflection near the origin. We obtain a non trivial scaling limit which behaves like reflected Brownian motion until its local time at zero reaches an exponential variable. It then follows reflected Brownian motion on the other side of the origin until its local time at zero reaches another exponential level, etc. These random walks are used in population genetics to trace the position of ancestors in the past near geographical barriers.
			
		\textbf{Keywords: }
		\newcommand{\sep}{;}
		generalised Brownian motions\sep random walks with obstacles\sep partial reflection\sep barriers to gene flow. \\
		\textbf{2010~MSC: } 60G99, 60F99, 60J65, 60J70
	\end{abstract}

	\section*{Introduction}
	
	Barriers to gene flow are physical obstacles to migration.
	Examples include mountain ranges, highways, political borders and the Great Wall of China \cite{su_great_2003}.
	All these geographical features leave traces in the genetic composition of populations living on both sides of the barrier.
	Geneticists try to use these traces to detect barriers to gene flow and to quantify their effect on migration.
	
	A naive approach to this problem would be to compute a measure of genetic differentiation (\textit{e.g.} Wright's F\textsubscript{st} \cite{wright_isolation_1943} which measures the level of inbreeding in the population resulting from its structure \cite{slatkin_comparison_1989}) between the two populations on each side of the candidate barrier, and to say that the latter acts as a barrier to gene flow if two individuals living on the same side are more related to each other on average than two indivudals living on different sides of the barrier.
	
	This method assumes that the two subpopulations on each side of the obstacle are well mixed.
	This may not always be a reasonable assumption and in some cases it is preferable to take into account the finer scale geographic structure of the sampled population.
	
	Mathematical models for spatially extended populations with barriers to gene flow already exist in the literature \cite{nagylaki_clines_1976,slatkin_gene_1973}, but most assume a discrete space and finding analytical formulae in this framework is challenging at best.
	Such formulae are particularly useful for inference purposes, where computational power is limiting.
	This paper is a step towards a rigorous mathematical framework to model genetic isolation by distance with barriers to gene flow in a continuous space.
	
	\paragraph*{Stepping stone model with a barrier at the origin}
	
	Nagylaki and his co-authors proposed the following model for the evolution of a spatially structured population with a barrier to gene flow \cite{nagylaki_clines_1976,nagylaki_clines_2012,nagylaki_clines_2016}.
	Consider a population living in a discrete linear space, with colonies (or demes) at locations $ \lbrace \ldots, -2, -1, 1, 2, \ldots \rbrace $.
	Each deme contains $ N $ individuals, and at each generation, those individuals are replaced by the offspring of the previous generation.
	An individual in deme $ i \notin \{-1, 1 \} $ has its parent in the previous generation in deme $ i-1 $ or $ i+1 $ with probability $ m/2 $ for some $ m \in (0,1) $, otherwise its parent is drawn from deme $ i $.
	Individuals in deme $ 1 $ have their parent in deme $ 2 $ with probability $ m/2 $ and in deme $ -1 $ with probability $ cm/2 $, with $ c \in (0,1) $, and likewise individuals in deme $ -1 $ have their parent in deme $ 1 $ with probability $ cm /2 $.
	Migration probabilities are depicted in Figure~\ref{fig:nagylaki}.
	Properties of this model and applications to various settings were studied in a sequence of papers \cite{nagylaki_clines_1976,nagylaki_influence_1988,nagylaki_influence_1993,nagylaki_clines_2012-1,nagylaki_clines_2016-1}.
	
	\begin{figure}[h]
		\centering
		\includegraphics[width=.9\textwidth]{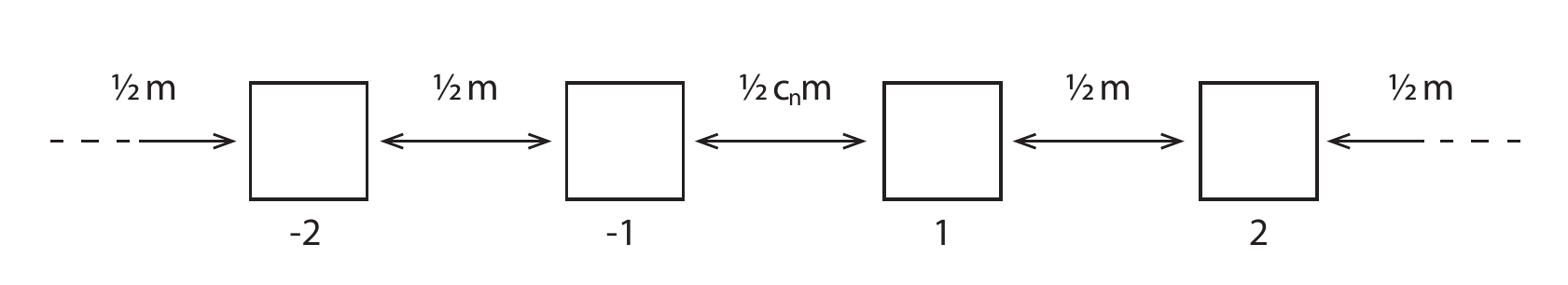}
		\caption{Stepping stone model with a barrier to gene flow} \label{fig:nagylaki}
	\end{figure}
	
	In this model, two individuals sampled at a given distance from each other will be more related if they are sampled on the same side of the origin than if they are not.
	This can be seen by assuming that each new individual mutates to a type never seen before with some probability $ \mu \in (0,1) $.
	Relatedness between individuals can then be measured by the probability that two sampled individuals are of the same type.
	This probability is called the probability of identity by descent, and it is given by the probability generating function of the age of the most recent common ancestor of these individuals.
	Properties of this function were studied in this setting in \cite{barton_effect_2008}.
	
	Also of interest is the evolution of the frequency of a given type (or allele) in the population, denoted by $ ( p^N(t,x), x \in \Z \setminus \{0\}, t \geq 0 ) $.
	Ignoring mutations and setting
	\begin{equation*}
	 \bar{p}(t,x) = \E{p^N(t,x)},
	\end{equation*}
	$\bar{p}$ solves a simple difference equation.
	Setting
	\begin{equation*}
		\bar{p}_n(t,x) = \bar{p}(nt, \sqrt{n}x)
	\end{equation*}
	for $ n \geq 1 $ and assuming that $ \sqrt{n} c \to \gamma \in (0,\infty) $, Nagylaki \cite{nagylaki_clines_1976} showed that $ \bar{p}_n $ converges to the solution of the following equation
	\begin{equation} \label{equation_allele_frequencies}
	\begin{cases}
	\partial_t p(t,x) = \frac{\sigma^2}{2} \partial_{xx} p(t,x) \qquad \text{ for } x \in \R \setminus \lbrace 0 \rbrace,\\
	\partial_x p(t,0^-) = \partial_x p(t,0^+) = \gamma ( p(t,0^+) - p(t,0^-) )
	\end{cases}
	\end{equation}
	where $ \sigma^2 = m $, see Figure~\ref{fig:allele_frequency}.
	In \cite{nagylaki_influence_1988} (see also \cite{barton_effect_2008}), Nagylaki showed a similar approximation for the probability of identity by descent.
	
	\begin{figure}[h]
		\centering
		\includegraphics[width=\textwidth]{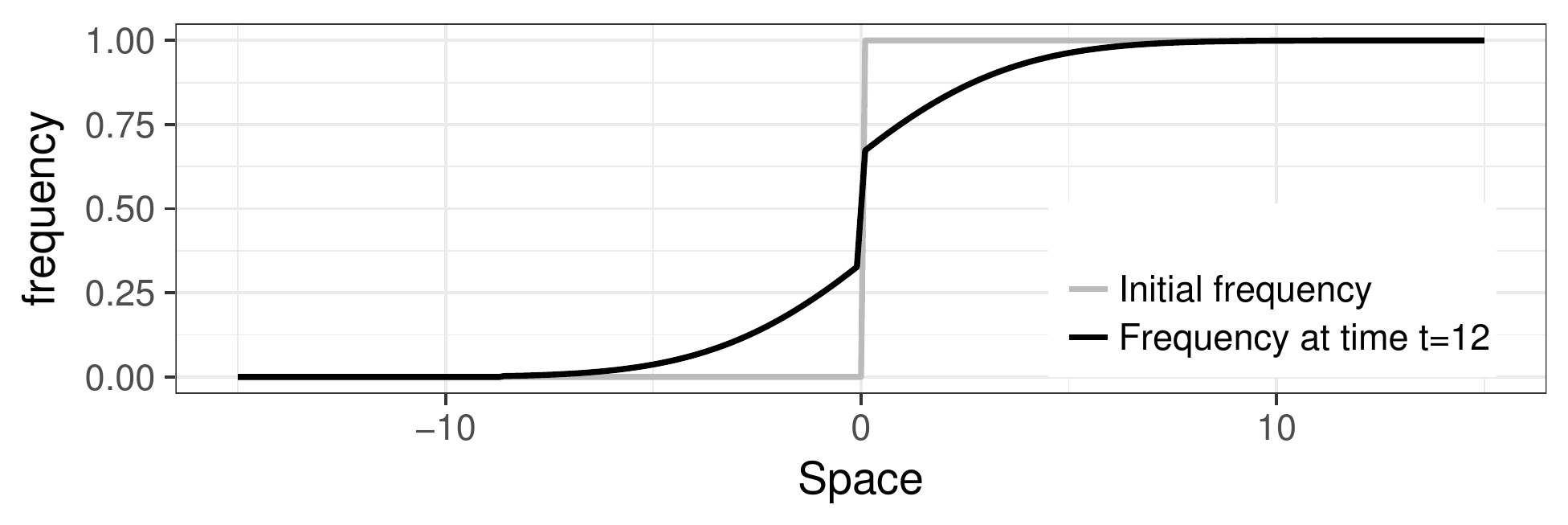}
		\caption{Evolution of allele frequency}
		\label{fig:allele_frequency}
		{\small Frequency of a type initially present at the right of the origin after a few generations.}
	\end{figure}
	
	\paragraph*{Duality}
	
	An alternative way to study this model from the forwards in time evolution of types is to look back in time for the position of one's ancestor some number of generations in the past.
	If $ \xi^x_t $ denotes the position of the ancestor $ t $ generations ago of an individual sampled at $ x \in \Z \setminus \lbrace 0 \rbrace $, then $ ( \xi^x_t, t \geq 0) $ is a random walk with transition probabilities given by the migration rates in Figure~\ref{fig:nagylaki}.
	In the absence of mutations, the proportion of individuals carrying a given allele at location $ x $ is the proportion of those individuals whose ancestor $ t $ generations ago carried the same allele.
	As a result, for $ x \in \Z \setminus \{0\} $ and $ t \geq 0 $,
	\begin{equation*}
	\bar{p}(t,x) = \E{p^N(0,\xi^x_t)}.
	\end{equation*}
	
	Likewise, the probability of identity by descent can be expressed with the help of the coalescence time of two random walks $ \xi^x, \xi^y $, \textit{i.e.} the first time that the two ancestors have the same parent.
	
	\paragraph*{Scaling limits of random walks with obstacles}
	
	In this paper, we present a result on the scaling limits of a class of \rw* with obstacles which includes $ (\xi^x_t, t\geq 0) $.
	For $ n \geq 1 $, if we set
	\begin{equation*}
	X_n(t) = \frac{1}{\sqrt{n}} \xi_{nt}^{\lfloor \sqrt{n} x \rfloor},
	\end{equation*}
	and if $ c $ is of order $ n^{-1/2} $, we show that $ X_n $ converges in distribution to a continuous stochastic process.
	This process resembles \Bm everywhere except near the origin where it has a singular behaviour.
	More precisely, this process behaves like reflected \Bm until its local time at the origin reaches an exponential \rv, after which it becomes reflected \Bm on the other side of the origin, until its local time reaches another exponential variable, and so on.
	We call this process \prBm.
	It generalises elastic \Bm considered for example in \cite{grebenkov_partially_2006}.
	
	The same process was obtained as a limit of one dimensional diffusions in \cite{mandrekar_brownian_2016}. 
	For $ \varepsilon > 0 $, they consider $ (X_\varepsilon(t), t\geq 0) $, solution to
	\begin{equation*}
	d X_\varepsilon(t) = \frac{L_\varepsilon}{\varepsilon} a \left( \frac{1}{\varepsilon} X_\varepsilon(t) \right) dt + d B_t,
	\end{equation*}
	where $ B $ is standard \Bm, $ L_\varepsilon \to \infty $ as $ \varepsilon \downarrow 0 $ and $ \sign(x) a(x) \geq 0 $, $ supp(a) \subset [-1, 1] $, and they give conditions on $ L_\varepsilon $ under which $ X_\varepsilon $ converges to \prBm (which they call \Bm with a hard membrane).
	
	This process also appears in \cite{lejay_snapping_2016} under the name \textit{snapping out \Bm} and is obtained as a limit of one dimensional diffusions with a small diffusivity in a thin layer around the origin.
	Lejay also gives another construction by piecing together a sequence of elastic \Bm*, choosing their sign at random each time the process is killed and reborn.
	
	In addition, we give a different construction of \prBm inspired by the speed and scale construction of one dimensional diffusions.
	Starting with standard \Bm, we glue together its excursions above level $ x > 0 $ and below level $ -x $ and we show that the result is the same process as the one described above.
	
	Moreover, we provide a martingale problem characterisation of \prBm, where equation \eqref{equation_allele_frequencies} can be seen as the action of the semigroup of \prBm on the initial allele frequency.
	In particular, the domain of the infinitesimal generator associated to \prBm is precisely the space of twice continuously differentiable functions on $ \R \setminus \lbrace 0 \rbrace $ satisfying
	\begin{equation*}
	\deriv{p}{x}(0^+) = \deriv{p}{x}(0^-) = \gamma (p(0^+) - p(0^-))
	\end{equation*}
	for some $ \gamma > 0 $.
	
	We also provide an explicit formula for the transition density of \prBm in Corollary~\ref{cor:transition_density} below.
	It turns out that this transition density has already been in use in the field of diffusion in porous media \cite{novikov_random_2011,grebenkov_exploring_2014}, but without mention of the underlying stochastic process.
	
	More recently, this process has been used in \cite{ringbauer_estimating_2018} to detect barriers to gene flow in genetic samples and to measure their strength from the resulting distortion in isolation by distance patterns.
	
	This paper is laid out as  follows.
	In Section~\ref{sec:results}, we present our main results: \prBm is defined as the solution to a martingale problem and two constructions of this process are given, we also state the convergence of a class of \rw* to \prBm.
	In Section~\ref{sec:construction}, we prove that the martingale problem which charaterizes \prBm is well posed and we show that the two constructions in Section~\ref{sec:results} provide solutions to this martingale problem.
	Finally in Section~\ref{sec:convergence}, we prove the convergence in distribution of a sequence of \rw* to \prBm.
	
	\section*{Acknowledgements}
	
	The author would like to thank Amandine Véber for suggesting this interesting problem and for helpfull discussions and comments all along this project.
	The author is also indebted to Alison Etheridge, \'Etienne Pardoux and Steven N. Evans whose comments and suggestions greatly improved the presentation of the results.
	Finally, the author would like to thank two anonymous referees who made a number of constructive remarks and helped improve this paper.
	
	The author was supported in part by the chaire Modélisation Mathématique et Biodiversité of Veolia Environnement-École Polytechnique-Museum National d’Histoire Naturelle-Fondation X.
	The author declares no competing interest.
	
	\section{Main results} \label{sec:results}
	
	\subsection{Definition and constructions of \prBm} \label{subsec:definition}
	
	We first give a definition of \prBm as a solution to a martingale problem on a space consisting of the disjoint union of the positive and negative half lines.
	We will show that this martingale problem is well posed.
	We then give two constructions of this process.
	
	\paragraph*{Definition}
	
	Let $ \ddot{\R} $ be the disjoint union of the positive and negative half lines,
	\begin{equation*}
	\ddot{\R} = (-\infty,0^-] \cup [0^+, +\infty).
	\end{equation*}
	It is endowed with the metric $ d $ defined by
	\begin{equation*}
	\forall x, y \in \ddot{\R}, \quad d(x,y) = \abs{x-y} + \1{xy \leq 0}.
	\end{equation*}
	Let $ \hat{C}(\ddot{\R}) $ be the set of continuous real-valued functions on $ \ddot{\R} $ which vanish at infinity.
	For $ \gamma \in [0,+\infty] $, let $ \mathcal{D}^\gamma $ be the subspace of functions $ f \in \hat{C}(\ddot{\R}) $ which are twice continuously differentiable on each half line and satisfy
	\begin{equation} \label{transmission_condition}
	\deriv{f}{x}(0^-) = \deriv{f}{x}(0^+) = \gamma ( f(0^+) - f(0^-) ).
	\end{equation}
	(For $ \gamma = +\infty $, \eqref{transmission_condition} becomes $ f'(0^-) = f'(0^+) $ and $ f(0^-) = f(0^+) $.)
	Fix $ \sigma > 0 $ and let us define a linear operator $ L^\gamma $ on $ \mathcal{D}^\gamma $ by
	\begin{equation} \label{definition_Lgamma}
	L^\gamma f = \frac{\sigma^2}{2} \deriv[2]{f}{x}, \quad \forall f \in \mathcal{D}^\gamma.
	\end{equation}
	Let $ D ( \R_+, \ddot{\R} ) $ denote the space of \cadlag functions from $ \R_+ $ to $ \ddot{\R} $.
	
	\begin{definition}[\prBm] \label{definition_prBm}
		Let $ \proc{X} $ be a \cadlag, $ \ddot{\R} $-valued Markov process on a probability space $ \left( \Omega, \A, \p \right) $, and call $ \p^X $ its law on $ D(\R_+, \ddot{\R}) $.
		The process $ \proc{X} $ (resp. its law $ \p^X $) is said to be (the law of) \prBm if it is a solution to the martingale problem associated with $ L^\gamma $ for some $ \gamma \in [0,+\infty] $, \textit{i.e.} if for any $ f \in \mathcal{D}^\gamma $, the process
		\begin{equation} \label{martingale_pb}
		f(X_t) - \int_{0}^{t} L^\gamma f(X_s) ds
		\end{equation}
		is a martingale with respect to the filtration generated by $ \proc{X} $.
		We call $ \gamma $ the \textit{permeability} of the barrier. 
	\end{definition}
	
	Naturally, we say that $ \proc{X} $ is \prBm with initial distribution $ \mu $ if it is a solution to the martingale problem associated with $ ( L^\gamma, \mu) $, \textit{i.e.} if \eqref{martingale_pb} is a martingale for all $ f \in \mathcal{D}^\gamma $ and if $ \p^X \left( X_0 \in \cdot \right) = \mu(\cdot) $.
	
	The operator $ L^\gamma $ is thus the generator of \prBm.
	This definition does not seem to provide much information about possible solutions to this martingale problem.
	It does not even tell us if such solutions exist or if they are unique (in distribution).
	This is the subject of the next proposition.
	Below, we also give two ways to construct solutions to this martingale problem.
	
	It should be noted that for $ \gamma = 0 $ (impermeable barrier), the operator $ L^\gamma $ is the generator of reflected Brownian motion (see for example Exercice VII.1.23 in \cite{revuz_continuous_2013} in the case $ \alpha=1 $), while for $ \gamma = +\infty $ (completely permeable barrier), $ L^\gamma $ is the generator of \Bm.
	
	\begin{proposition} \label{propo:well_posedness}
		For any $ \gamma \in [0,+\infty] $, the martingale problem associated with $ L^\gamma $ has at most one $ D(\R_+, \ddot{\R}) $ valued solution.
	\end{proposition}
	
	\begin{proof}
		The operator $ L^\gamma $ satisfies the positive maximum principle on $ \mathcal{D}^\gamma $, \textit{i.e.} whenever $ f \in \mathcal{D}^\gamma $ and $ \sup_{x \in \ddot{\R}} f(x) = f(x_0) \geq 0 $, we have $ L^\gamma f(x_0) \leq 0 $.
		By Lemma~4.2.1 of \cite{ethier_markov_1986}, $ L^\gamma $ is thus dissipative on $ \mathcal{D}^\gamma $ (recall that we require the functions in $ \mathcal{D}^\gamma $ to vanish at infinity).
		
		Let us now show that for any positive $ \lambda $, the range of $ \lambda - L^\gamma $ contains the space $ \hat{C}(\ddot{\R}) $ of continuous functions vanishing at infinity.
		We do it in the case $ \sigma^2 = 2 $, but the general case is similar.
		Let $ f \in \hat{C}(\ddot{\R}) $ be such a function and define
		\begin{equation*}
		g(x) = \left\lbrace
		\begin{aligned}
		& e^{-\sqrt{\lambda} x} \int_{0}^{x} \dfrac{e^{\sqrt{\lambda} y}}{2 \sqrt{\lambda}} f(y) dy + e^{\sqrt{\lambda} x} \int_{x}^{+\infty} \dfrac{e^{-\sqrt{\lambda} y}}{2 \sqrt{\lambda}} f(y) dy + A e^{-\sqrt{\lambda} x} & \text{ if } x \geq 0^+, \\
		& e^{- \sqrt{\lambda} x} \int_{-\infty}^{x} \dfrac{e^{\sqrt{\lambda} y}}{2\sqrt{\lambda}} f(y) dy + e^{\sqrt{\lambda} x} \int_{x}^{0} \dfrac{e^{-\sqrt{\lambda} y}}{2\sqrt{\lambda}} f(y) dy + B e^{\sqrt{\lambda} x} & \text{ if } x \leq 0^-,
		\end{aligned}
		\right.
		\end{equation*}
		for some $ A, B \in \R $.
		Then $ g $ is twice continuously differentiable on $ \ddot{\R} $, vanishes at infinity and satisfies
		\begin{equation*}
		\deriv[2]{g}{x}(x) = \lambda g(x) - f(x)
		\end{equation*}
		for all $ x \in \ddot{\R} $.
		The constants $ A $ and $ B $ can then be chosen so that $ g $ also satisfies \eqref{transmission_condition}.
		As a result we have found a function $ g $ in $ \mathcal{D}^\gamma $ such that $ \lambda g - L^\gamma g = f $ for any $ f \in \hat{C}(\ddot{\R}) $.
		
		In particular, since $ \mathcal{D}^\gamma $ is a subset of $ \hat{C}(\ddot{\R}) $, it is in the range of $ \lambda - L^\gamma $ for any $ \lambda > 0 $.
		Furthermore $ \hat{C}(\ddot{\R}) $ is separating in the sense of Section~3.4 in \cite{ethier_markov_1986}.
		Proposition~\ref{propo:well_posedness} then follows from Corollary~4.4.4 in \cite{ethier_markov_1986}. 
	\end{proof}
	
	\paragraph*{"Speed and scale" construction of \prBm}
	
	We now present a way to construct \prBm from \Bm, via an analogy with the speed and scale construction of one dimensional diffusions.
	This will give us a better sense of what "typical" trajectories of this process look like.
	Indeed, we show that the excursions of \prBm outside the origin are given by the sequence of excursions of a \Bm outside a macroscopic region of length $ \frac{1}{\gamma} $, as illustrated in Figure~\ref{fig:speed_scale}.
	
	\begin{figure}[h]
		\centering
		\includegraphics[width=.9\textwidth]{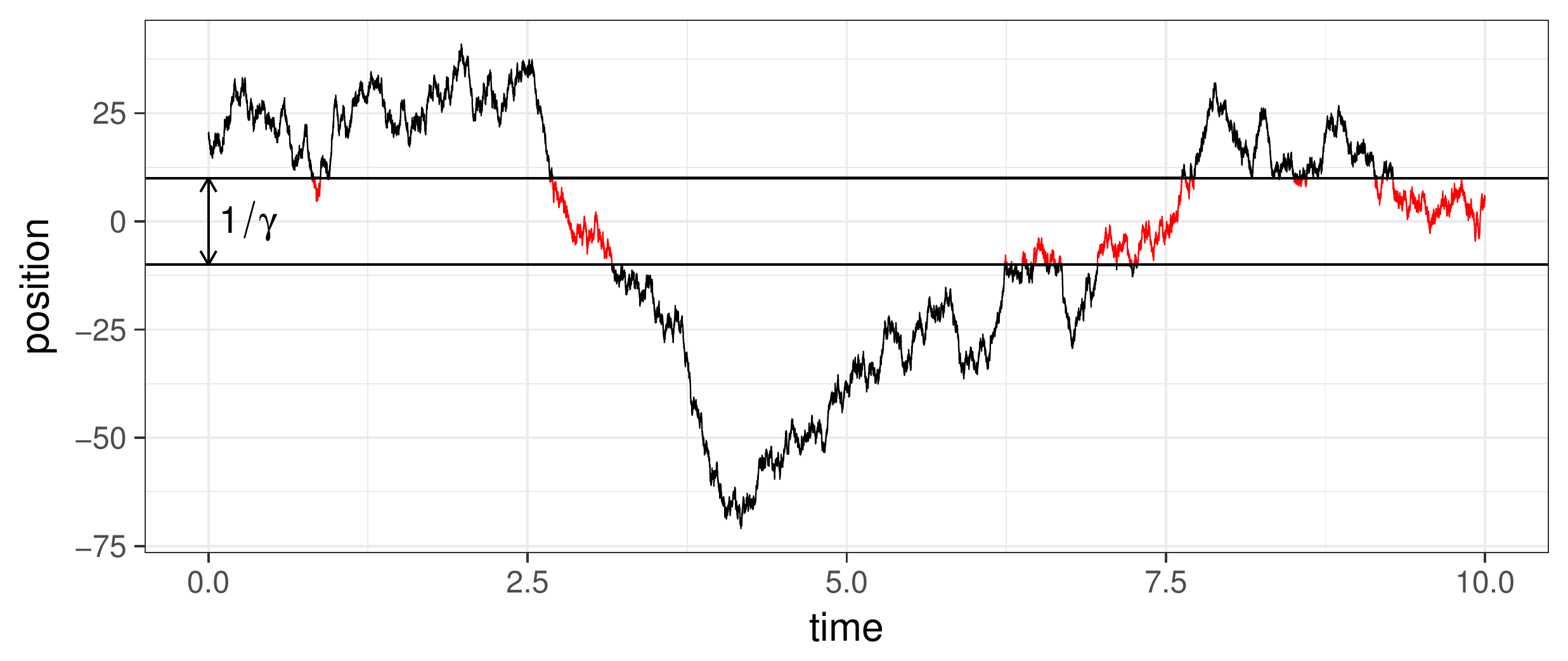}
		
		\includegraphics[width=.9\textwidth]{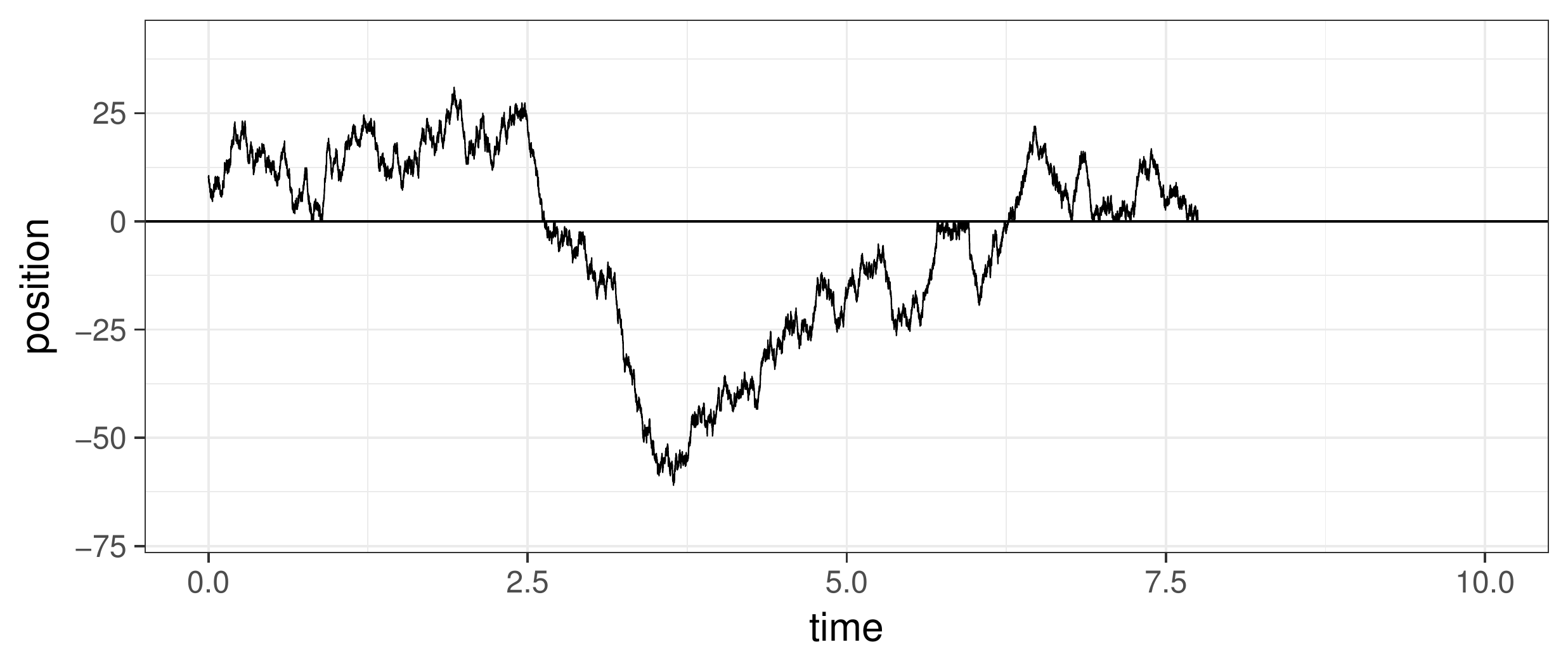}
		\caption{Speed and scale construction of \prBm}
		\label{fig:speed_scale}
		{\small Construction of \prBm as a time-changed \Bm, with $ \sigma^2 = 400 $, $ x = 20 $ and $ \gamma = 0.05 $.}
	\end{figure}
	
	\begin{wrapfigure}{R}{0.4\textwidth}
		\centering
		%\vspace{-1em}
		\includegraphics[width=0.39\textwidth]{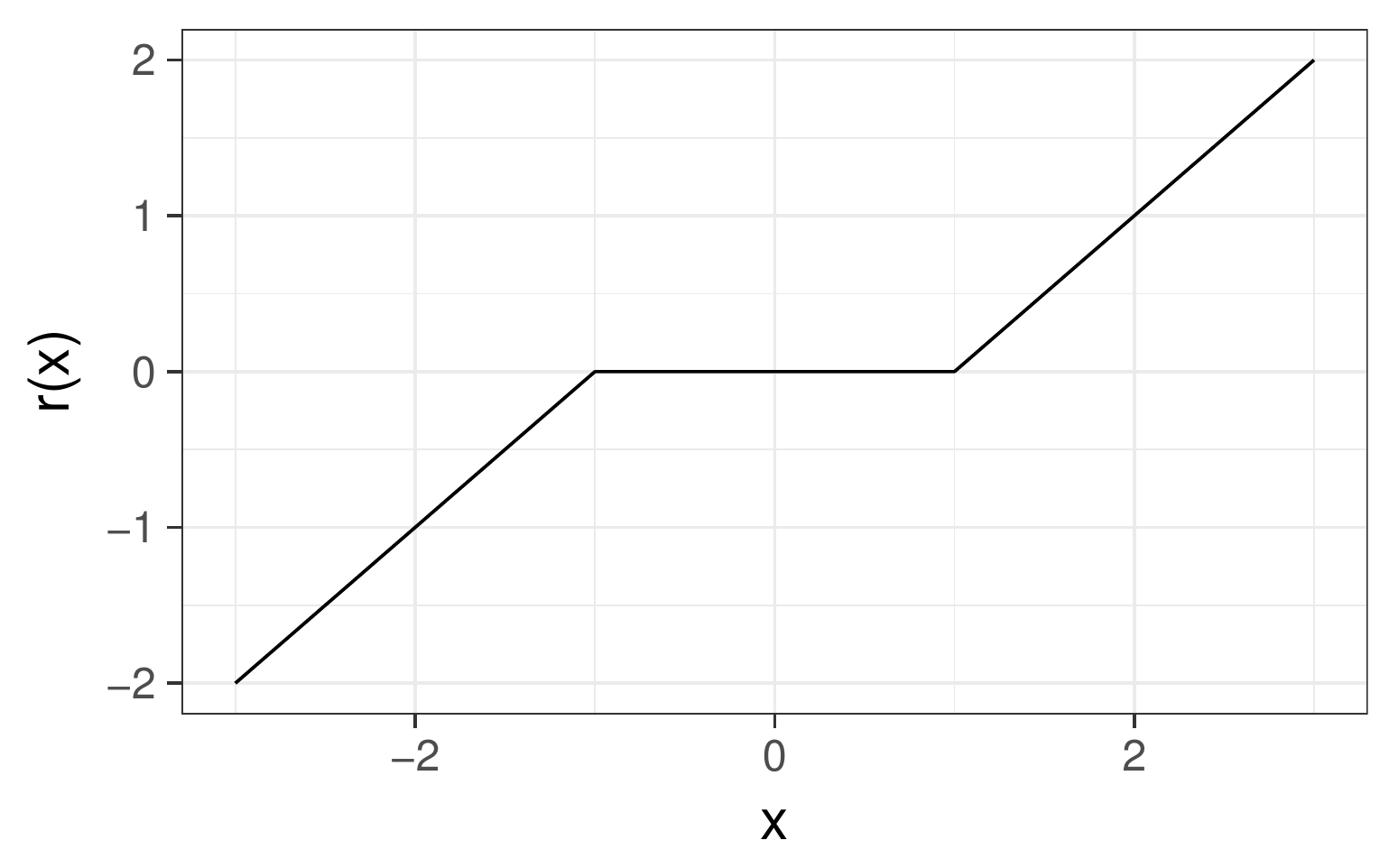}
		\caption{Graph of the function $ r : \R \to \ddot{\R} $ for $ \gamma = 0.5 $.} \label{fig:scale_function}
		%\vspace{-2.5em}
	\end{wrapfigure}
	
	Fix $ \gamma \in (0, +\infty) $ and suppose for simplicity that $ \sigma^2 = 1 $.
	Define $ r : \R \to \ddot{\R} $ by
	\begin{equation*}
	r(x) = \begin{cases}
	x - \frac{1}{2\gamma} & \text{ if } x > \frac{1}{2\gamma}, \\
	x + \frac{1}{2\gamma} & \text{ if } x < - \frac{1}{2\gamma}, \\
	0^+ & \text{ if } 0 \leq x \leq \frac{1}{2\gamma}, \\
	0^- & \text{ if } - \frac{1}{2\gamma} \leq x < 0,
	\end{cases}
	\end{equation*}
	(see Figure~\ref{fig:scale_function}).
	Further define $ r^{-1} : \ddot{\R} \to \R $ by
	\begin{equation*}
	r^{-1}(x) = x + \sign(x) \frac{1}{2\gamma}.
	\end{equation*}
	(Note that $ r^{-1} $ is only the right inverse of $ r $, \textit{i.e.} $ r \circ r^{-1} = Id_{\ddot{\R}} $ but $ r^{-1} \circ r \neq Id_\R $.)
	Now fix $ x \in \ddot{\R} $ and let $ \proc{B} $ be standard \Bm started from $ r^{-1}(x) $.
	Also set, for $ t \geq 0 $,
	\begin{equation} \label{def:tau}
	\tau(t) = \inf \left\lbrace \tau > 0 : \int_{0}^{\tau} \1{ \abs{B_s} > \frac{1}{2\gamma}} ds > t \right\rbrace.
	\end{equation}
	Finally, let $ X_t = r(B_{\tau(t)}) $.
	
	\begin{proposition} \label{prop:speed_scale}
		The process $ \proc{X} $ is \prBm started from $ x $, \textit{i.e.} it is a solution to the martingale problem associated with $ \left( L^\gamma, \delta_x \right) $.
	\end{proposition}
	
	We prove this result in Subsection~\ref{subsec:speed_scale}.
	The construction is illustrated in Figure~\ref{fig:speed_scale}.
	In words, we map the two intervals $ (-\infty, - \frac{1}{2\gamma}] $ and $ [\frac{1}{2\gamma}, +\infty) $ onto $ (-\infty, 0^-] $ and $ [0^+, +\infty) $, and we change time in order to drop the time intervals where $ \abs{B_s} \leq \frac{1}{2\gamma} $.
	
	\begin{remark}
		From this construction, we recover the property stated by Nagylaki \citep[Equation~56]{nagylaki_clines_1976} that, for $ 0 < x < y $,
		\begin{equation*}
		\P[x]{ X_t \text{ reaches $ 0^- $ before $ y $} } = \frac{y - x}{y + \frac{1}{\gamma}}.
		\end{equation*}
		More generally, Proposition~\ref{prop:speed_scale} shows that $ r^{-1} $ is the scale function of \prBm.
	\end{remark}
	
	\begin{cor}
		For any $ \gamma \in [0,+\infty] $, the martingale problem associated to $ L^\gamma $ is well posed, \textit{i.e.} it has a unique solution.
	\end{cor}
	
	Let $ \ddot{\pi} : \ddot{\R} \to \R $ be the natural projection of $ \ddot{\R} $ onto $ \R $ (\textit{i.e.} mapping both $ 0^+ $ and $ 0^- $ onto $ 0 $).
	We sometimes also call the projection $ \proc{\ddot{\pi}(X_t)}{t\geq 0} $ \prBm, even though the latter isn't a Markov process.
	(For example, as we shall see below, the sequence of \rw* considered in Subsection~\ref{subsec:scaling_limit} converges to the projection of \prBm.)
	
	\paragraph*{Construction involving the local time at the origin}
	
	From the previous construction, one is led to think that $ \proc{\abs{X_t}}{t\geq 0} $ has the law of reflected \Bm.
	It is then natural to ask if \prBm can be constructed by randomly "flipping" the excursions of reflected \Bm.
	The next proposition provides such a construction.
	It turns out that the crossing times of the origin are the times at which the local time at the origin of the process reaches the arrival times of an independent Poisson process with parameter $ \gamma $, as in Figure~\ref{fig:construction_local_time}.
	
	Fix $ x \in \ddot{\R} $ and let $ \proc{W} $ be reflected \Bm on $ \R_+ $ started from $ \abs{x} $.
	Let $ \left( N(t), t\geq 0 \right) $ be a Poisson process with rate $ \gamma \in (0,\infty) $, independent of $ \proc{W} $.
	Let $ L^0_t(W) $ denote the local time accumulated at the origin up to time $ t $ by $ W $, that is,
	\begin{equation*}
		L^0_t(W) = \lim_{\varepsilon \downarrow 0} \frac{1}{2 \varepsilon} \int_{0}^{t} \1{\abs{W_s} \leq \varepsilon} ds
	\end{equation*}
	(see \citep[Chapter~VI]{revuz_continuous_2013}).
	Set
	\begin{equation*}
	X_t = \sign(x) (-1)^{N(L^0_t(W))} W_t,
	\end{equation*}
	where $ \pm 1 \times 0 = 0^\pm $ (see Figure~\ref{fig:construction_local_time}).
	
	\begin{figure}[h]
		{\centering
			\includegraphics[width=\textwidth]{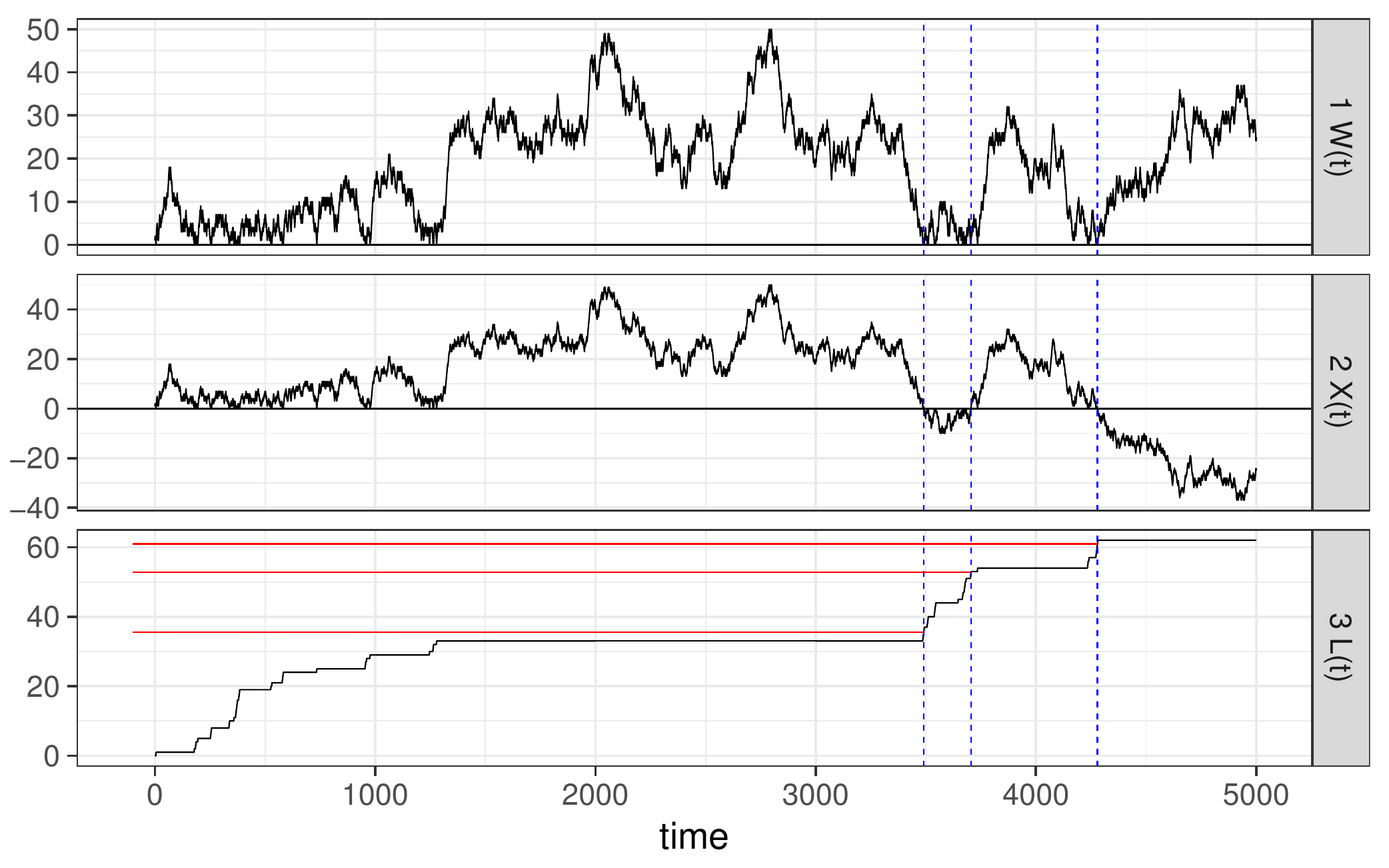}
			\caption{Construction of partially reflected \Bm involving the local time at the origin}
			\label{fig:construction_local_time}}
		{\small Top graph shows a realisation of reflected Brownian motion $ W_t $. Bottom graph shows its local time accumulated at the origin $ L(t) $. The heights of horizontal red lines are drawn according to a Poisson process on the y axis. The graph in the middle is obtained by "flipping" $ W_t $ at the times when $ L(t) $ reaches the red lines.
		The corresponding process is distributed as the projection of \prBm.}
	\end{figure}
	
	\begin{proposition} \label{prop:local_time_construction}
		The process $ \proc{X} $ is \prBm started from $ x $.
	\end{proposition}
	
	We prove this result in Subsection~\ref{subsec:local_time} using the previous construction and a Ray-Knight theorem \citep[Theorem 6.4.7]{shreve_brownian_1991}, which states that the local time accumulated by Brownian motion at $ \frac{1}{2\gamma} $ before reaching $ -\frac{1}{2\gamma} $ is an exponential \rv with parameter $ \gamma $.
	
	Proposition~\ref{prop:local_time_construction} yields an explicit formula for the transition density of \prBm.
	For $ t>0 $ and $ x \in \R $, set
	\begin{equation*}
	G_t(x) = \frac{1}{\sqrt{2\pi t}} \exp \left( -\frac{x^2}{2 t} \right).
	\end{equation*}
	
	\begin{cor} \label{cor:transition_density}
		If $ \proc{X} $ is \prBm with permeability $ \gamma \in (0,\infty) $ started from $ x \in \ddot{\R} $, then $ \P[x]{X_t \in dy} = g_t(x,y) dy $ with
		\begin{equation*}
		g_t(x,y) = \begin{cases}
		G_t(x-y) + G_t(x+y) - 2\gamma \int_{0}^{+\infty} e^{-2\gamma l} G_t( \abs{x} + \abs{y} + l ) dl & \text{ if } xy \geq 0^+, \\
		2\gamma \int_{0}^{+\infty} e^{-2\gamma l} G_t( \abs{x} + \abs{y} + l ) dl & \text{ if } xy \leq 0^-.
		\end{cases}
		\end{equation*}
	\end{cor}
	
	We derive this formula in Subsection~\ref{subsec:transition_density}.
	
	\subsection{Scaling limits of a class of \rw*} \label{subsec:scaling_limit}
	
	Let us now state the main convergence result, namely that \prBm is the scaling limit of a class of \rw* with an obstacle.
	We consider a more general case than \cite{nagylaki_clines_1976}, where the barrier to gene flow has width $ K \in \N \setminus \lbrace 0 \rbrace $ ($ K $ being the number of edges with reduced migration rate).
	The cases $ K=1 $ (the one considered in \cite{nagylaki_clines_1976}) and $ K=2 $ are illustrated in Figure~\ref{fig:jump_rates}.
	We define the process describing the motion of an ancestral lineage as follows.
	
	\begin{figure}[h]
		\centering
		a) \includegraphics[align=c,width=.9\textwidth]{migration_rates_k1-eps-converted-to.pdf}
		
		b) \includegraphics[align=c,width=.9\textwidth]{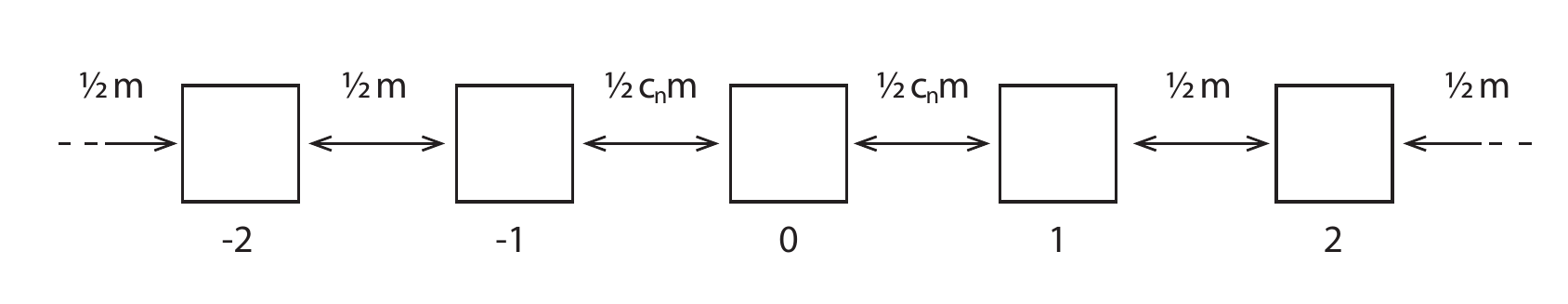}
		\caption{Jump rates of random walks with an obstacle}
		\label{fig:jump_rates}
		{\small Transition rates of the random walk $ (\xi_n(t), t \geq 0) $ in a) case $ K=1 $ and $ b) $ case $ K=2 $.}
	\end{figure}
	
	\begin{definition}[Random walk with an obstacle] \label{definition:rw_obstacle}
		Let $ \proc{c_n}{n \geq 1} $ be a sequence of positive real numbers, and fix $ m>0 $.
		Suppose that $ \proc{x_n^0}{n\geq 1} $ is a sequence of elements of $ \N^* $.
		\subparagraph*{If $ K $ is even,} let $ E = \Z $ and define, for $ i \neq j \in E $,
		\begin{equation} \label{jump_rates_even}
		q_n(i,j) = \begin{cases}
		\frac{m}{2} & \text{ if } \abs{i-j} = 1 \text{ and } \abs{i} \vee \abs{j} > \frac{K}{2}, \\
		c_n \frac{m}{2} & \text{ if } \abs{i - j} = 1 \text{ and } \abs{i} \vee \abs{j} \leq \frac{K}{2}, \\
		0 & \text{ otherwise.}
		\end{cases}
		\end{equation}
		\subparagraph*{If $ K $ is odd,} let $ E = \Z \setminus \lbrace 0 \rbrace $ and set, for $ i \neq j \in E $,
		\begin{equation} \label{jump_rates_odd}
		q_n(i,j) = \begin{cases}
		\frac{m}{2} & \text{ if } \abs{i-j} = 1 \text{ and } \abs{i} \vee \abs{j} > \frac{K+1}{2}, \\
		c_n \frac{m}{2} & \text{ if } \abs{i - j} = 1 \text{ and } \abs{i} \vee \abs{j} \leq \frac{K+1}{2}, \\
		& \text{ or if }  \lbrace i, j \rbrace = \lbrace +1, -1 \rbrace,  \\
		0 & \text{ otherwise.}
		\end{cases}
		\end{equation}
		Then let $ \left( \xi_n(t), t \geq 0 \right) $ be a continuous time \rw on $ E $ started from $ x_n^0 $ with jump rates $ q_n(\cdot,\cdot) $.
		That is, whenever $ \xi_n(t) = i $, the future of the random walk is determined as follows.
		Attach to each $ j \neq i $ in $ E $ such that $ q_n(i,j) > 0 $ an independent exponential random variable $ E_j $ with parameter $ q_n(i,j) $.
		Then at time $ J_1 = t + \inf \{ E_j \} $, the random walk jumps to state $ k = \argmin_j \{ E_j \}  $.
		This procedure is then repeated between each jump time of the random walk.
	\end{definition}
	
	For $ n\geq 1 $, set $ X_n(t) = \frac{1}{\sqrt{n}} \xi_n(nt) $.
	We now state conditions under which the rescaled \rw $ X_n = (X_n(t), t \geq 0) $ converges to \prBm.
	We equip the space of \cadlag functions from $ \R_+ $ to $ \R $ (denoted by $ \sko[\R_+]{\R} $) with the topology of Skorokhod convergence on compact time intervals.
	If $ d_{sko}^T( \cdot, \cdot ) $ is a metric for the Skorokhod convergence on $ \sko{\R} $ for $ T>0 $, then
	\begin{equation} \label{definition_metric}
	d_{sko}(f,g) = \int_{0}^{+\infty} e^{-t} ( d_{sko}^t(f,g) \wedge 1 ) dt
	\end{equation}
	is a metric for Skorokhod convergence on compact time intervals.
	
	\begin{theorem} \label{thm:convergence}
		Suppose $ \frac{1}{\sqrt{n}} x_n^0 \cvgas{n} x^0 $ with $ x^0 \neq 0 $ and $ \lim_{n\to \infty} \frac{\sqrt{n}}{K} c_n = \gamma \in [0,+\infty] $.
		Then as $ n \to \infty $, the sequence of real-valued processes $ \left( X_n(t), t\geq 0 \right) $ converges in distribution in $ ( \sko[\R_+]{\R}, d_{sko} ) $ to a continuous real-valued process $ (X(t), t \geq 0) $ which is (a projection on $ \R $ of) a solution to the martingale problem associated with $ \left( L^\gamma, \delta_{x^0} \right) $, with $ \sigma^2 = m $.
	\end{theorem}
	
	In other words, if $ \sqrt{n} c_n \to +\infty $, $ X_n $ converges to Brownian motion, if $ \sqrt{n} c_n \to 0 $, $ X_n $ converges to reflected \Bm, while if $ \frac{\sqrt{n}}{K} c_n \to \gamma \in (0,\infty) $, $ X_n $ converges to (the projection of) \prBm (recall that the latter takes values in $ \ddot{\R} $, its projection is obtained by identifying $ 0^+ $ and $ 0^- $ with $ 0 $).
	
	\begin{remark}
		In the case $ x^0 = 0 $, the convergence still holds provided the probability of first exiting the set $ [ - K/2, K/2 ] $ on the right converges as $ n \to \infty $.
		The initial distribution is then a convex combination of $ \delta_{0^+} $ and $ \delta_{0^-} $, given by the limits of the exit probabilities.
	\end{remark}
	
	Theorem~\ref{thm:convergence} is proved in Section~\ref{sec:convergence} in the case $ K=2 $.
	The generalisation to other values of $ K $ is straighforward, and the case $ K=1 $ introduces some simplifications, which makes the case $ K=2 $ more representative of the general case.
	
	Note that in Nagylaki's model presented in Figure~\ref{fig:nagylaki}, ancestral lineages are distributed as the \rw of Definition~\ref{definition:rw_obstacle} with $ K=1 $.
	
	\section{Constructions of \prBm} \label{sec:construction}
	
	\subsection{Speed and scale construction} \label{subsec:speed_scale}
	
	Here we prove that the process $ X_t = r(B_{\tau(t)}) $ defined in Subsection~\ref{subsec:definition} is a solution to the martingale problem associated with $ L^\gamma $.
	This proof will require the following lemma, proved in Subsection~\ref{subsec:speed_scale_absolute_value}.
	
	\begin{lemma} \label{lemma:speed_scale_absolute_value}
		Set $ W_t = \abs{X_t} $.
		Then $ \proc{W} $ is distributed as reflected \Bm.
	\end{lemma}
	
	\begin{proof}[Proof of Proposition~\ref{prop:speed_scale}]
		Recall that $ B $ is standard \Bm started at $ r^{-1}(x) $, hence $ X_0 = x $ almost surely.
		Let $ \proc{\F^B} $ denote the natural filtration of $ \proc{B} $, and let $ \F_t = \F^B_{\tau(t)} $.
		Then $ \proc{\F} $ is a filtration, $ \proc{X} $ is $ \proc{\F} $ adapted and, for $ s, t \geq 0 $ and $ f : \ddot{\R} \to \R $ bounded and continuous,
		\begin{equation*}
		\E{f(X_{t+s})}{\F_t} = \E[r^{-1}(X_t)]{f(r(B_{\tau(s)}))}.
		\end{equation*}
		Hence $ (X_t, t \geq 0) $ is a Markov process with respect to $ (\F_t, t \geq 0) $.
		Now let $ \proc{\F^X} $ be the filtration generated by $ \proc{X} $.
		Since $ \F^X_t \subset \F_t $, $ \proc{X} $ is a Markov process with respect to $ (\F^X_t, t \geq 0) $.
		
		Suppose now that for any $ x \in \ddot{\R} $ and $ f \in \mathcal{D}^\gamma $,
		\begin{equation} \label{inf_generator}
		\lim_{t\downarrow 0} \dfrac{1}{t} \E[x]{f(X_t)-f(X_0)} = \dfrac{1}{2} \deriv[2]{f}{x}(x).
		\end{equation}
		(Recall that we assumed $ \sigma^2 = 1 $ for simplicity.)
		Then, by Proposition~4.1.7 in \cite{ethier_markov_1986}, \eqref{martingale_pb} is an $ \F^X $-martingale for all $ f \in \mathcal{D}^\gamma $ ($ X $ is progressive since it is right-continuous).
		It follows that $ \proc{X} $ is a solution to the martingale problem associated with $ L^\gamma $.
		
		Let us now show \eqref{inf_generator}.
		Since $ X $ behaves as standard \Bm until the first time it hits the origin, \eqref{inf_generator} clearly holds for all $ x \in \ddot{\R} \setminus \lbrace 0^+, 0^- \rbrace $.
		By symmetry, we can restrict the proof to $ x=0^+ $.
		For any $ t\geq 0 $,
		\begin{multline*}
		\E[0^+]{f(X_t)} = \E[0^+]{\left( f(X_t)-f(0^+) \right) \1{X_t\geq 0^+} + \left( f(X_t)-f(0^-) \right) \1{X_t \leq 0^-}} \\ + \E[0^+]{f(0^+) \1{X_t \geq 0^+} + f(0^-) \1{X_t\leq 0^-}}.
		\end{multline*}
		Subtracting $ f(0^+) $ on both sides we obtain
		\begin{multline} \label{generator_step_1}
		\E[0^+]{f(X_t) - f(0^+)} = \E[0^+]{\left( f(X_t)-f(0^+) \right) \1{X_t\geq 0^+} + \left( f(X_t)-f(0^-) \right) \1{X_t \leq 0^-}} \\ + \E[0^+]{\left( f(0^-)-f(0^+) \right) \1{X_t \leq 0^-}}.
		\end{multline}
		Since $ f $ is twice continuously differentiable on $ [0^+, +\infty) $, for any $ y $ in $ [0^+, +\infty) $, there exists $ h(y) \in [0^+, y] $ such that
		\begin{equation*}
		f(y) - f(0^+) = \deriv{f}{x}(0^+) y + \frac{1}{2} \deriv[2]{f}{x}(h(y)) y^2.
		\end{equation*}
		Replacing $ y $ by $ X_t $, we write, for any $ r > 0 $,
		\begin{multline*}
		(f(X_t) - f(0^+)) \1{X_t \geq 0^+} = \left( \deriv{f}{x}(0^+) X_t + \frac{1}{2} \deriv[2]{f}{x}(h(X_t)) X_t^2 \right) \1{0^+ \leq X_t \leq r} \\ + (f(X_t) - f(0^+)) \1{X_t > r}.
		\end{multline*}
		By the Markov inequality and Lemma~\ref{lemma:speed_scale_absolute_value},
		\begin{equation} \label{bound_exit}
		\P{\abs{X_t} > r} \leq 3 \frac{t^2}{r^4}.
		\end{equation}
		As a result, since $ f $ is bounded,
		\begin{equation}
		\E[0^+]{(f(X_t) - f(0^+)) \1{X_t > r}} \leq 6 \norm[\infty]{f} \frac{t^2}{r^4}.
		\end{equation}
		In addition,
		\begin{align} \label{error_1}
		\E[0^+]{\deriv{f}{x}(0^+) X_t \1{0^+ \leq X_t \leq r}} = \deriv{f}{x}(0^+) \E[0^+]{X_t \1{X_t \geq 0^+}} - \deriv{f}{x}(0^+) \E[0^+]{X_t \1{X_t > r}},
		\end{align}
		and by the Cauchy-Schwartz inequality, Lemma~\ref{lemma:speed_scale_absolute_value} and \eqref{bound_exit},
		\begin{align*}
		\abs{ \E[0^+]{X_t \1{X_t > r}} } &\leq \E[0^+]{X_t^2}^{1/2} \P[0^+]{X_t > r}^{1/2} \\
		& \leq t^{1/2} \sqrt{3} \frac{t}{r^2}. \numberthis \label{remainder_1}
		\end{align*}
		Moreover, since $ f'' $ is continuous on $ [0^+, +\infty) $, it is uniformly continuous on compact sets and there exists $ C_r > 0 $ such that
		\begin{equation*}
		\forall x, y \in [0^+, r], \qquad \abs{\deriv[2]{f}{x}(y) - \deriv[2]{f}{x}(x)} \leq C_r \abs{x-y}.
		\end{equation*}
		As a result,
		\begin{equation} \label{error_2}
		\abs{ \E[0^+]{\deriv[2]{f}{x}(h(X_t)) X_t^2 \1{0^+ \leq X_t \leq r}} - \E[0^+]{\deriv[2]{f}{x}(0^+) X_t^2 \1{0\leq X_t \leq r}} } \leq C_r \E[0^+]{\abs{X_t}^3},
		\end{equation}
		and by Lemma~\ref{lemma:speed_scale_absolute_value}, $ \E[0^+]{\abs{X_t}^3} = \bigO{t^{3/2}} $.
		Proceeding as for \eqref{remainder_1}, we also have
		\begin{equation} \label{error_3}
		\E[0^+]{\frac{1}{2} \deriv[2]{f}{x}(0^+) X_t^2 \1{0^+ \leq X_t \leq r}} = \frac{1}{2} \deriv[2]{f}{x}(0^+) \E[0^+]{X_t^2 \1{X_t \geq 0^+}} + \bigO{t^{3/2}}.
		\end{equation}
		Putting together \eqref{error_1}, \eqref{error_2} and \eqref{error_3}, we obtain
		\begin{multline*}
		\E[0^+]{(f(X_t) - f(0^+)) \1{X_t \geq 0^+}} = \deriv{f}{x}(0^+) \E[0^+]{X_t \1{X_t \geq 0^+}} \\ + \frac{1}{2} \deriv[2]{f}{x}(0^+) \E[0^+]{X_t^2 \1{X_t \geq 0^+}} + \littleO{t}.
		\end{multline*}
		Likewise, we have
		\begin{multline*}
		\E[0^+]{(f(X_t)-f(0^-)) \1{X_t \leq 0^-}} = \deriv{f}{x}(0^-) \E[0^+]{X_t \1{X_t \leq 0^-}} \\ + \frac{1}{2} \deriv[2]{f}{x}(0^-) \E[0^+]{X_t^2 \1{X_t \leq 0^-}} + \littleO{t}.
		\end{multline*}
		Plugging these two equations in \eqref{generator_step_1} and using the fact that $ f'(0^-) = f'(0^+) $, we obtain
		\begin{multline} \label{speed_scale_taylor}
		\E[0^+]{f(X_t)-f(X_0)} = \deriv{f}{x}(0^\pm) \E[0^+]{X_t} + \dfrac{1}{2} \deriv[2]{f}{x}(0^+) \E[0^+]{X_t^2} \\ + \dfrac{1}{2} \left( \deriv[2]{f}{x}(0^-) - \deriv[2]{f}{x}(0^+) \right) \E[0^+]{X_t^2 \1{X_t \leq 0^-}} \\ + \left( f(0^-) - f(0^+) \right) \P[0^+]{X_t \leq 0^-} + \littleO{t}.
		\end{multline}
		Moreover, by the construction of $ X_t $,
		\begin{align*}
		\E[0^+]{X_t} &= \E[\frac{1}{2\gamma}]{r(B_{\tau(t)})} \\
		&= \E[\frac{1}{2\gamma}]{\left( B_{\tau(t)} - \dfrac{1}{2\gamma} \right) \1{B_{\tau(t)} \geq \frac{1}{2\gamma}} + \left( B_{\tau(t)} + \dfrac{1}{2\gamma} \right) \1{B_{\tau(t)} \leq - \frac{1}{2\gamma}}} \\
		&= \E[\frac{1}{2\gamma}]{B_{\tau(t)}} + \dfrac{1}{2\gamma} \E[\frac{1}{2\gamma}]{\1{B_{\tau(t)}\leq - \frac{1}{2\gamma}} - \1{B_{\tau(t)}\geq \frac{1}{2\gamma}}}. \numberthis \label{expectation_Xt}
		\end{align*}
		Note that $ \tau(t) $ is an $ \F^B_t $-stopping time.
		Furthermore, for any given $ t\geq 0 $, the martingale $ ( B_{s \wedge \tau(t)} , s \geq 0 ) $ is uniformly integrable.
		To see this, write
		\begin{equation*}
		\sup_{s \geq 0} \abs{B_{s \wedge \tau(t)}} \leq \frac{1}{2\gamma} + \sup_{0 \leq s \leq t} W_s,
		\end{equation*}
		and note that the \rhs is integrable by Lemma~\ref{lemma:speed_scale_absolute_value} and Doob's maximal inequality.
		Hence, by the Optional Stopping Theorem, $ \E[\frac{1}{2\gamma}]{B_{\tau(t)}} = \frac{1}{2\gamma} $.
		As a result, returning to \eqref{expectation_Xt},
		\begin{equation*}
		\E[0^+]{X_t} = \frac{1}{\gamma} \P[0^+]{X_t \leq 0^-}.
		\end{equation*}
		Since $ f \in \mathcal{D}^\gamma $, the first term in \eqref{speed_scale_taylor} cancels with the last one.
		By Lemma~\ref{lemma:speed_scale_absolute_value}, $ \E[0^+]{X_t^2} = t $.
		Also note that by the Cauchy-Schwarz inequality,
		\begin{equation*}
		\E[0^+]{X_t^2 \1{X_t \leq 0^-}} \leq \sqrt{3} \, t \, \P[0^+]{X_t \leq 0^-}^{1/2}.
		\end{equation*}
		(We have used Lemma~\ref{lemma:speed_scale_absolute_value} to compute the fourth moment of $ X_t $.)
		Furthermore, $$ \P[0^+]{X_t \leq 0^-} = \P[\frac{1}{2\gamma}]{B_{\tau(t)} \leq -\frac{1}{2\gamma}} \cvgas{t}{0} 0. $$
		
		Coming back to \eqref{speed_scale_taylor}, dividing both sides by $ t $ and letting $ t \downarrow 0 $, we obtain
		\begin{equation*}
		\lim_{t \downarrow 0} \frac{1}{t} \E[0^+]{f(X_t) - f(X_0)} = \frac{1}{2} \deriv[2]{f}{x}(0^+).
		\end{equation*}
		The proof of Proposition~\ref{prop:speed_scale} is now complete.
	\end{proof}
	
	\subsection{Construction involving the local time at the origin} \label{subsec:local_time}
	
	Let $ \proc{B} $ be standard \Bm and let $ X_t = r(B_{\tau(t)}) $ be \prBm constructed as before.
	Set $W_t = \abs{X_t}$ and
	\begin{align*}
	T_0 = 0, && T_{i+1} = \inf \lbrace t > T_i : X_{T_i} X_t < 0 \rbrace, \quad i \geq 0.
	\end{align*}
	For $ i \geq 1 $ set
	\begin{equation*}
	E_i = L^0_{T_i}(X) - L^0_{T_{i-1}}(X)
	\end{equation*}
	and for $ t \geq 0 $,
	\begin{equation*}
	N(t) = \max \left\lbrace n \in \N : \sum_{i=0}^{n} E_i \leq t \right\rbrace.
	\end{equation*}
	Then, for all $ t \geq 0 $,
	\begin{equation*}
	X_t = \sign(X_0) (-1)^{N(L^0_t(W))} W_t.
	\end{equation*}
	We know from Lemma~\ref{lemma:speed_scale_absolute_value} that $\proc{W}$ is distributed as reflected \Bm.
	Proposition~\ref{prop:local_time_construction} will be proven if we show that $ (N(t), t \geq 0) $ is a Poisson process with rate $ \gamma $ and that it is independent of $ \proc{W} $. 
	
	The fact that $ N $ and $ W $ are independent might seem implausible at first sight as they are both constructed from $ \proc{B} $.
	However, the $ E_i $ (and hence $ N $) only depend on the amount of local time that $ B $ accumulates at $ \pm \frac{1}{2\gamma} $ between successive crossings of $ [-\frac{1}{2\gamma}, \frac{1}{2\gamma}] $, and those crossing times cannot be determined by observing $ \proc{W} $.
	To prove this, we construct two independent processes in such a way that $ \proc{W} $ is a function of the former and $ (N(t), t \geq 0) $ is a function of the latter.
	Set
	\begin{equation*}
	\theta(t) = \inf \left \lbrace \theta > 0 : \int_{0}^{\theta} \1{ \abs{B_s} \leq \frac{1}{2\gamma} } ds > t \right\rbrace.
	\end{equation*}
	We prove the following in Subsection~\ref{subsec:speed_scale_absolute_value}.
	
	\begin{lemma} \label{lemma:independence}
		The processes $ (\abs{B_{\tau(t)}}, t \geq 0) $ and $ (B_{\theta(t)}, t\geq 0) $ are independent.
	\end{lemma}
		
	\begin{proof}[Proof of Proposition~\ref{prop:local_time_construction}]
		Note that the left (resp. right) local time accumulated by $X$ at the origin up to time $t$ is the local time accumulated by $B$ at $-\frac{1}{2\gamma}$ (resp. $\frac{1}{2\gamma}$) up to time $\tau(t)$.
		Indeed, by the Tanaka formula \cite[Theorem~VI.1.2]{revuz_continuous_2013}, letting $x^+ = \max(x,0)$,
		\begin{equation*}
		\frac{1}{2} L^{0^+}_t(X) = X_t^+ - X_0^+ - \int_{0}^{t} \1{X_s>0} d X_s
		\end{equation*}
		and
		\begin{equation*}
		\frac{1}{2} L^{1/2\gamma}_{\tau(t)}(B) = (B_{\tau(t)} - \frac{1}{2\gamma})^+ - (B_0 - \frac{1}{2 \gamma})^+ - \int_{0}^{\tau(t)} \1{B_s>\frac{1}{2\gamma}} d B_s.
		\end{equation*}
		(For \Bm, considering the right, the left or the symmetric local time makes no difference.)
		By the construction of $X$, $X_t^+ = (B_{\tau(t)} - \frac{1}{2\gamma})^+$ and, since $\1{B_s>\frac{1}{2\gamma}} = 0$ when $s \in (\tau(t^-), \tau(t))$,
		\begin{equation*}
		\int_{0}^{\tau(t)} \1{B_s>\frac{1}{2\gamma}} d B_s = \int_0^t \1{X_s>0} dX_s.
		\end{equation*}
		As a result, 
		\begin{equation} \label{local_time_XB}
		L^{0^+}_t(X) = L^{1/2\gamma}_{\tau(t)}(B)
		\end{equation}
		and likewise, $L^{0^-}_t(X) = L^{-1/2\gamma}_{\tau(t)}(B)$.
		
		For $ a \in \R $, set $ \mathds{T}_a = \inf\lbrace t > 0 : B_t = a \rbrace $.
		Assuming \Wlog that $X_0 > 0$, $ \tau(T_1) = \mathds{T}_{-1/2\gamma} $.
		Then,
		\begin{align*}
		E_1 = L^0_{T_1}(X) &= \frac{1}{2} \left( L^{0^+}_{T_1}(X) + L^{0^-}_{T_1}(X) \right) \\
		&= \frac{1}{2} L^{1/2\gamma}_{\mathds{T}_{-1/2\gamma}}(B).
		\end{align*}
		By the Ray-Knight theorem \cite[Theorem~6.4.7]{shreve_brownian_1991}, $$L^{1/2\gamma}_{\mathds{T}_{-1/2\gamma}}(B)$$ is an exponential \rv with parameter $\frac{\gamma}{2}$.
		Hence $E_1$ is exponential with parameter $\gamma$.
		
		Further, the strong Markov property of $\proc{B}$ and its symmetry imply that the $E_i$ are \iid.
		As a result $ (N(t), t \geq 0) $ is a Poisson process with rate $ \gamma $.
		It remains to show that it is independent of $ \proc{W} $.
		
		Define
		\begin{align*}
		S_0 = 0, && S_i = \inf \left\lbrace t > S_{i-1} : B_{\theta(t)} = \frac{(-1)^i}{2\gamma} \right\rbrace.
		\end{align*}
		By the same argument as above,
		\begin{equation*}
		L^0_{T_i}(W) = L^{1/2\gamma}_{S_i}(B_{\theta}) + L^{-1/2\gamma}_{S_i}(B_{\theta}).
		\end{equation*}
		As a result, the $ E_i $, and $ (N(t), t\geq 0) $, are measurable \wrt the sigma field generated by $ (B_\theta(t), t\geq 0) $.
		Since $ W_t = \abs{B_{\tau(t)}} - \frac{1}{2\gamma} $, Lemma~\ref{lemma:independence} implies the independence of $ N $ and $ W $.
		This concludes the proof of Proposition~\ref{prop:local_time_construction}.
	\end{proof}
	
	\subsection{The absolute value of \prBm} \label{subsec:speed_scale_absolute_value}
	
	Let us start by recalling the following lemma, due to Skorokhod \cite{skorokhod_stochastic_1961} (also Lemma~3.6.14 in \cite{shreve_brownian_1991}).
	
	\begin{lemma}[\cite{skorokhod_stochastic_1961}] \label{lemma:skorokhod}
		Let $ f : \R_+ \to \R $ be a continuous function with $ f(0) \geq 0 $.
		There exists a unique continuous function $ l : \R_+ \to \R $ such that
		\begin{enumerate}[i)]
			\item $ X(t) := l(t) + f(t) $ is non negative for all $ t \geq 0 $,
			\item $ l(0) = 0 $ and $ t \mapsto l(t) $ is non decreasing,
			\item $ \int_{0}^{\infty} \1{X(t) > 0} dl(t) = 0 $.
		\end{enumerate}
		The function $ l $ is then called the solution of the Skorokhod problem for $ f $ and it is given by
		\begin{equation*}
			l(t) = \inf_{0\leq s \leq t} (f(s))^-.
		\end{equation*}
	\end{lemma}
	
	The following generalisation can be found in \citep[Proposition~2.4.6]{harrison_brownian_1985}.
	
	\begin{lemma}[\cite{harrison_brownian_1985}] \label{lemma:two_sided_regulator}
		Fix $ a < b \in \R $ and let $ f : \R_+ \to \R $ be a continuous function such that $ f(0) \in [a,b] $.
		There exists a unique pair of continuous functions $ (l, u) $ from $ \R_+ $ to $ \R $ such that
		\begin{enumerate}[i)]
			\item $ X(t) := f(t) + l(t) - u(t) \in [a,b] $ for all $ t \geq 0 $,
			\item $ l(0) = u(0) = 0 $ and $ l $ and $ u $ are non decreasing,
			\item $ \int_{0}^{\infty} \1{X(t) > a} dl(t) = \int_{0}^{\infty} \1{X(t) < b} du(t) = 0 $.
		\end{enumerate}
		The pair $ (l, u) $ is called the two-sided regulator of $ f $.
	\end{lemma}
	
	For $ t \geq 0 $, set
	\begin{align*}
	I^1(t) &= \int_{0}^{t} \1{B_s > \frac{1}{2\gamma}} d B_s - \int_{0}^{t} \1{B_s < - \frac{1}{2\gamma}} dB_s, \\
	I^2(t) &= \int_{0}^{t} \1{\abs{B_s} \leq \frac{1}{2\gamma}} d B_s.
	\end{align*}
	Both $ I^1 $ and $ I^2 $ are continuous $ \F^B_t $ martingales with
	\begin{align*}
	&\qvar{I^1} = \int_{0}^{t} \1{ \abs{B_s} > \frac{1}{2\gamma} } ds \\
	&\qvar{I^2} = \int_{0}^{t} \1{ \abs{B_s} \leq \frac{1}{2\gamma} } ds \\
	&\qvar{I^1, I^2} = 0.
	\end{align*}
	By F. B. Knight's theorem \cite{knight_reduction_1971} (also Theorem~3.4.13 in \cite{shreve_brownian_1991}), the processes
	\begin{align*}
	\tilde{B}^1_t = W_0 + I^1(\tau(t)), && \tilde{B}^2_t = B_{\theta(0)} + I^2(\theta(t)),
	\end{align*}
	are independent standard \Bm*.
	
	\begin{proof}[Proof of Lemma~\ref{lemma:speed_scale_absolute_value}]
		By the Tanaka formula \citep[Theorem~VI.1.2]{revuz_continuous_2013},
		\begin{align} \label{tanaka_1}
		& \frac{1}{2} L^{1/2\gamma}_t(B) = \left( B_t - \tfrac{1}{2\gamma} \right)^+ - \left( B_0 - \tfrac{1}{2\gamma} \right)^+ - \int_{0}^{t} \1{B_s > \frac{1}{2\gamma}} dB_s, \\ \label{tanaka_2}
		& \frac{1}{2} L^{-1/2\gamma}_t(B) = \left( B_t + \tfrac{1}{2\gamma} \right)^- - \left( B_0 + \tfrac{1}{2\gamma} \right)^- + \int_{0}^{t} \1{B_s < - \frac{1}{2\gamma}} dB_s.
		\end{align}
		On the other hand, from the construction of $X_t$,
		\begin{equation*}
		W_t = \abs{X_t} = (B_{\tau(t)}-\frac{1}{2\gamma})^+ + (B_{\tau(t)}+\frac{1}{2\gamma})^-
		\end{equation*}
		and from \eqref{local_time_XB},
		\begin{equation*}
		L^0_t(W) = L^0_t(X) = \frac{1}{2}\left( L^{1/2\gamma}_{\tau(t)}(B) + L^{-1/2\gamma}_{\tau(t)}(B) \right).
		\end{equation*}
		Adding \eqref{tanaka_1} and \eqref{tanaka_2} and replacing $ t $ by $ \tau(t) $, we obtain
		\begin{equation*}
		\tilde{B}^1_t = W_t - L^0_t(W).
		\end{equation*}
		Since $\tilde{B}^1$ is standard \Bm, $W$ is reflected \Bm \cite[VI.2]{revuz_continuous_2013}.
	\end{proof}
	
	\begin{proof}[Proof of Lemma~\ref{lemma:independence}]
		Since $ \tilde{B}^1_t = W_t - L^0_t(W) $, the function $ t \mapsto L^0_t(W) $ is a solution of the Skorokhod problem for $ t \mapsto \tilde{B}^1_t $, and by Lemma~\ref{lemma:skorokhod},
		\begin{equation*}
		W_t = \tilde{B}^1_t + \inf_{s \leq t} (\tilde{B}^1_s)^-.
		\end{equation*}
		On the other hand, $ B_{\theta(t)} $ is a function of $ ( \tilde{B}^2_t, t\geq 0) $.
		To see this, note that since $ B_{\theta(t)} \in [-1/2\gamma, 1/2\gamma] $,
		\begin{equation*}
		B_{\theta(t)} = \left( B_{\theta(t)} + \frac{1}{2\gamma} \right)^+ - \left( B_{\theta(t)} - \frac{1}{2\gamma} \right)^+ - \frac{1}{2\gamma}.
		\end{equation*}
		By the Tanaka formula,
		\begin{align*}
		\left( B_{t} + \frac{1}{2\gamma} \right)^+ &= \left( B_0 + \frac{1}{2\gamma} \right)^+ + \int_{0}^{t} \1{B_s > - \frac{1}{2\gamma}} d B_s + L^{-\frac{1}{2\gamma}}_t(B), \\
		\left( B_t - \frac{1}{2\gamma} \right)^+ &= \left( B_0 - \frac{1}{2\gamma} \right)^+ + \int_{0}^{t} \1{B_s > \frac{1}{2\gamma}} d B_s + L^{\frac{1}{2\gamma}}_t(B).
		\end{align*}
		Subtracting these equations with $ t $ replaced by $ \theta(t) $, and noting that $ \1{B_s \geq - 1/2\gamma} - \1{B_s > 1/2\gamma} = \1{\abs{B_s} \leq 1/2\gamma} $, we obtain
		\begin{equation*}
		B_{\theta(t)} = \tilde{B}^2_t + L^{-\frac{1}{2\gamma}}_{\theta(t)}(B) - L^{\frac{1}{2\gamma}}_{\theta(t)}(B).
		\end{equation*}
		From this equation, we see that $ ( L^{-\frac{1}{2\gamma}}_{\theta(\cdot)}(B), L^{\frac{1}{2\gamma}}_{\theta(\cdot)}(B) ) $ is the two-sided regulator of $ \tilde{B}^2 $ with reflection at $ \pm 1/2\gamma $.
		By Lemma~\ref{lemma:two_sided_regulator}, $ ( B_{\theta(t)}, t \geq 0 ) $ is then uniquely determined by $ ( \tilde{B}^2_t, t \geq 0) $.
		
		Since $ \abs{B_{\tau(t)}} = W_t + \frac{1}{2\gamma} $ is a function of $ \tilde{B}^1 $, $ B_{\theta(t)} $ is a function of $ \tilde{B}^2 $, and $ \tilde{B}^1 $ is independent of $ \tilde{B}^2 $, $ \left( \abs{B_{\tau(t)}}, t\geq 0 \right) $ and $ \left( B_{\theta(t)}, t \geq 0 \right) $ are independent.
	\end{proof}
	
	\subsection{Transition density of \prBm} \label{subsec:transition_density}
	
	\begin{proof}[Proof of Corollary~\ref{cor:transition_density}]
		Recall that $ X_t $ was defined as
		\begin{equation*}
		X_t = \sign(x) (-1)^{N(L^0_t(W))} W_t,
		\end{equation*}
		where $ W $ is reflected \Bm started from $ \abs{x} $ and $ \proc{N} $ is an independent Poisson process with rate $ \gamma $.
		Hence, summing over all possible values of $ L^0_t(W) $,
		\begin{equation*}
		\P[x]{X_t \in dy} = \int_{0}^{\infty} \P{N(l) \equiv \sign(x)-\sign(y) \: (\mathrm{mod}\, 2)} \P[\abs{x}]{W_t \in d\abs{y}, L^0_t(W) \in dl}.
		\end{equation*}
		Since $ N(l) $ is a Poisson \rv with parameter $ \gamma l $,
		\begin{align*}
		\P{N(l) \equiv 0 \:(\mathrm{mod}\, 2)} = \frac{1+e^{-2\gamma l}}{2}, && \P{N(l) \equiv 1 \:(\mathrm{mod}\, 2)} = \frac{1-e^{-2\gamma l}}{2}.
		\end{align*}
		In addition, taking $ \alpha = 1/2 $ in Corollary~3.3 of \cite{appuhamillage_occupation_2011}, we obtain, for $ x, y \geq 0 $,
		\begin{equation*}
		\P[x]{W_t \in dy, L^0_t(W) \in dl} = \left( G_t(x - y) - G_t(x+y) \right) dy \delta_0(dl) - 2 \partial_x G_t(x+y+l) dy dl.
		\end{equation*}
		As a result, if $ xy \geq 0^+ $,
		\begin{equation*}
		\P[x]{X_t \in dy} = \left( G_t(x - y) - G_t(x+y) \right) dy - 2 \int_{0}^{\infty} \frac{1 + e^{-2\gamma l}}{2} \partial_x G_t(\abs{x}+\abs{y}+l) dl dy.
		\end{equation*}
		Integrating by parts yields
		\begin{equation*}
		\frac{\P[x]{X_t \in dy}}{dy} = G_t(x - y) + G_t(x+y) - 2\gamma \int_{0}^{\infty} e^{-2\gamma l} G_t(\abs{x}+\abs{y}+l) dl.
		\end{equation*}
		Likewise if $ xy \leq 0^- $,
		\begin{align*}
		\frac{\P[x]{X_t \in dy}}{dy} &= -2 \int_{0}^{\infty} \frac{1-e^{-2\gamma l}}{2} \partial_x G_t( \abs{x} + \abs{y} + l) dl \\
		&= 2\gamma \int_{0}^{\infty} e^{-2\gamma l} G_t( \abs{x} + \abs{y} + l) dl.
		\end{align*}
		The proof of Corollary~\ref{cor:transition_density} is now complete.
	\end{proof}
	
	\section{Scaling limit of \rw* with a barrier} \label{sec:convergence}
	
	Here, we prove the convergence of the sequence of \rw* defined in Subsection~\ref{subsec:scaling_limit} to \prBm (Theorem~\ref{thm:convergence}), in the case $K=2$ and $\gamma \in (0,\infty)$ (the general case is treated similarly).
	
	\begin{figure}[h]
		\centering
		\includegraphics[width=\textwidth]{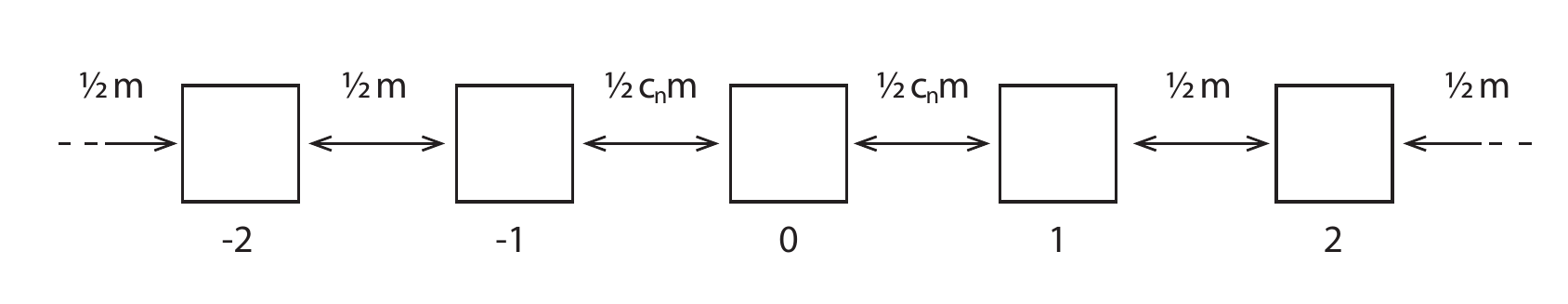}
		\caption{Jump rates of the \rw with an obstacle for $ K=2 $} \label{fig:jump_rates_k2}
	\end{figure}
	
	Recall that $\left( \xi_n(t), t\geq 0 \right)$ is a \rw on $E$ with jump rates given in \eqref{jump_rates_even}, \eqref{jump_rates_odd} (Figure~\ref{fig:jump_rates_k2}) and that $X_n(t) = \frac{1}{\sqrt{n}}\xi_n(nt)$.
	Also recall that $ d $ is a metric for Skorokhod convergence on compact time intervals \eqref{definition_metric}.
	\begin{lemma} \label{lemma:tightness}
		The sequence $\lbrace \proc{X_n(t)}{t\geq 0}, n \geq 1 \rbrace$ is tight in $ ( \sko[\R_+]{\R}, d ) $.
	\end{lemma}
	Let $\proc{X_\infty(t)}{t\geq 0}$ be an arbitrary limit point of this sequence (\textit{i.e.} the limit of a converging subsequence).
	\begin{lemma} \label{lemma:absolute_value}
		$\abs{X_\infty}$ is distributed as reflected \Bm with diffusion coefficient $m$.
	\end{lemma}
	Let $T_0=0$ and for $i \geq 0$,
	\begin{align} \label{def_Ti}
		T_{i+1} = \inf \lbrace t > T_i : X_\infty(T_i) X_\infty(t) < 0 \rbrace.
	\end{align}
	\begin{lemma} \label{lemma:local_time}
		$( L^0_{T_{i+1}}(X_\infty) - L^0_{T_{i}}(X_\infty) )_{i \geq 0}$ is a sequence of independent exponential \rv* with parameter $\gamma$.
		This sequence is independent of $\proc{\abs{X_\infty(t)}}{t\geq 0}$.
	\end{lemma}
	
	\begin{proof}[Proof of Theorem~\ref{thm:convergence}]
		By Proposition~\ref{prop:local_time_construction}, $X_\infty$ is characterized as (the projection on $\R$ of) \prBm.
		Since the sequence $X_n$ is tight and has only one possible limit point in $\sko[\R_+]{\R}$, it converges in distribution to \prBm.
	\end{proof}
	
	The rest of this section is devoted to the proof of Lemmas~\ref{lemma:absolute_value}, \ref{lemma:local_time} and \ref{lemma:tightness}, in that order.
	In what follows, we assume, with a slight abuse of notation, that $ \left( X_n, n\geq 1 \right) $ is a subsequence of the original sequence of processes which converges in distribution to $ X_\infty $.
	
	\subsection{The absolute value of $X_\infty$} \label{subsec:absolute_value}
	
	To prove that the absolute value of any possible limit point of $X_n$ is reflected \Bm, we write $\abs{X_n}$ as the sum of a martingale term and a non-decreasing term.
	We then show that the martingale term converges to \Bm while the non-decreasing term converges to the opposite of the running minimum of this \Bm.
	The conclusion follows from a classical result on reflected \Bm \cite[VI.2]{revuz_continuous_2013}.
	
	Set
	\begin{equation*}
	\tilde{X}_n(t) = \abs{X_n(t)} \1{\abs{X_n(t)} \geq \frac{2}{\sqrt{n}}}
	\end{equation*}
	and, for $i\geq 0$,
	\begin{align*}
	\sigma_0^n = 0, && \tau_i^n = \inf \lbrace t > \sigma_i^n : \abs{X_n(t)} \leq \tfrac{1}{\sqrt{n}} \rbrace, \\
	&& \sigma_{i+1}^n = \inf \lbrace t > \tau_i^n : \abs{X_n(t)} > \tfrac{1}{\sqrt{n}} \rbrace.
	\end{align*}
	The process $ \tilde{X}_n $ can then be decomposed as follows \cite{iksanov_functional_2016}
	\begin{equation} \label{tilde_decomposition}
	\tilde{X}_n(t) = M_n(t) + L_n(t) - \sum_{i\geq 0} \abs{X_n(\tau_i^n)} \1{ \tau_i^n \leq t < \sigma_{i+1}^n },
	\end{equation}
	with 
	\begin{equation*}
	M_n(t) = \abs{X_n(0)} + \int_0^t \1{ \abs{X_n(s)} > \frac{1}{\sqrt{n}} } d \abs{X_n}(s)
	\end{equation*}
	and
	\begin{align*}
	L_n(t) &= \sum_{i\geq 0} \left( \abs{X_n(\sigma_{i+1}^n)} - \abs{X_n(\tau_i^n)} \right) \1{ \sigma_{i+1}^n \leq t } \\
	&= \frac{1}{\sqrt{n}} \sum_{i\geq 0} \1{ \sigma_{i+1}^n \leq t }. \numberthis \label{local_time}
	\end{align*}
	The term $M_n$ is a martingale, while $L_n$ counts the number of visits (in fact of exits) of $[-\frac{1}{\sqrt{n}},\frac{1}{\sqrt{n}}]$.
	
	Define the running minimum $V_n(t)$ of the martingale part as
	\begin{equation} \label{definition_Vn}
	V_n(t) = \sup_{s \leq t} \left( \tfrac{2}{\sqrt{n}} - M_n(s) \right)^+
	\end{equation}
	and note that $V_n$ first becomes positive when $M_n$ first reaches $\frac{1}{\sqrt{n}}$, \textit{i.e.}
	\begin{equation*}
	\inf \lbrace t \geq 0 : V_n(t) \geq \tfrac{1}{\sqrt{n}} \rbrace = \inf \lbrace t \geq 0 : M_n(t) \leq \tfrac{1}{\sqrt{n}} \rbrace.
	\end{equation*}
	Since up to that time the other terms on the \rhs of \eqref{tilde_decomposition} are zero, we get
	\begin{equation*}
	\inf \lbrace t \geq 0 : V_n(t) \geq \tfrac{1}{\sqrt{n}} \rbrace = \tau_0^n.
	\end{equation*}
	The next time $V_n$ increases is
	\begin{equation*}
	\inf \lbrace t \geq 0 : V_n(t) \geq \tfrac{2}{\sqrt{n}} \rbrace = \inf \lbrace t \geq 0 : M_n(t) \leq 0 \rbrace.
	\end{equation*}
	By \eqref{tilde_decomposition}, this is also $\tau_1^n$.
	By induction,
	\begin{equation} \label{Vn_local_time}
	V_n(t) = \frac{1}{\sqrt{n}} \sum_{i\geq 0} \1{ \tau_i^n \leq t }.
	\end{equation}
	This translates the fact that the excursions of $M_n$ above its running minimum are given by the excursions of $\abs{X_n}$ above $\frac{1}{\sqrt{n}}$, see also Figure~\ref{fig:random_walk_decomposition}.
	Returning to \eqref{tilde_decomposition}, we have shown
	\begin{equation} \label{martingale_plus_increasing}
	\tilde{X}_n(t) = M_n(t) + V_n(t).
	\end{equation}
	
	\begin{figure}[h]
		{\centering
			\includegraphics[width=\textwidth]{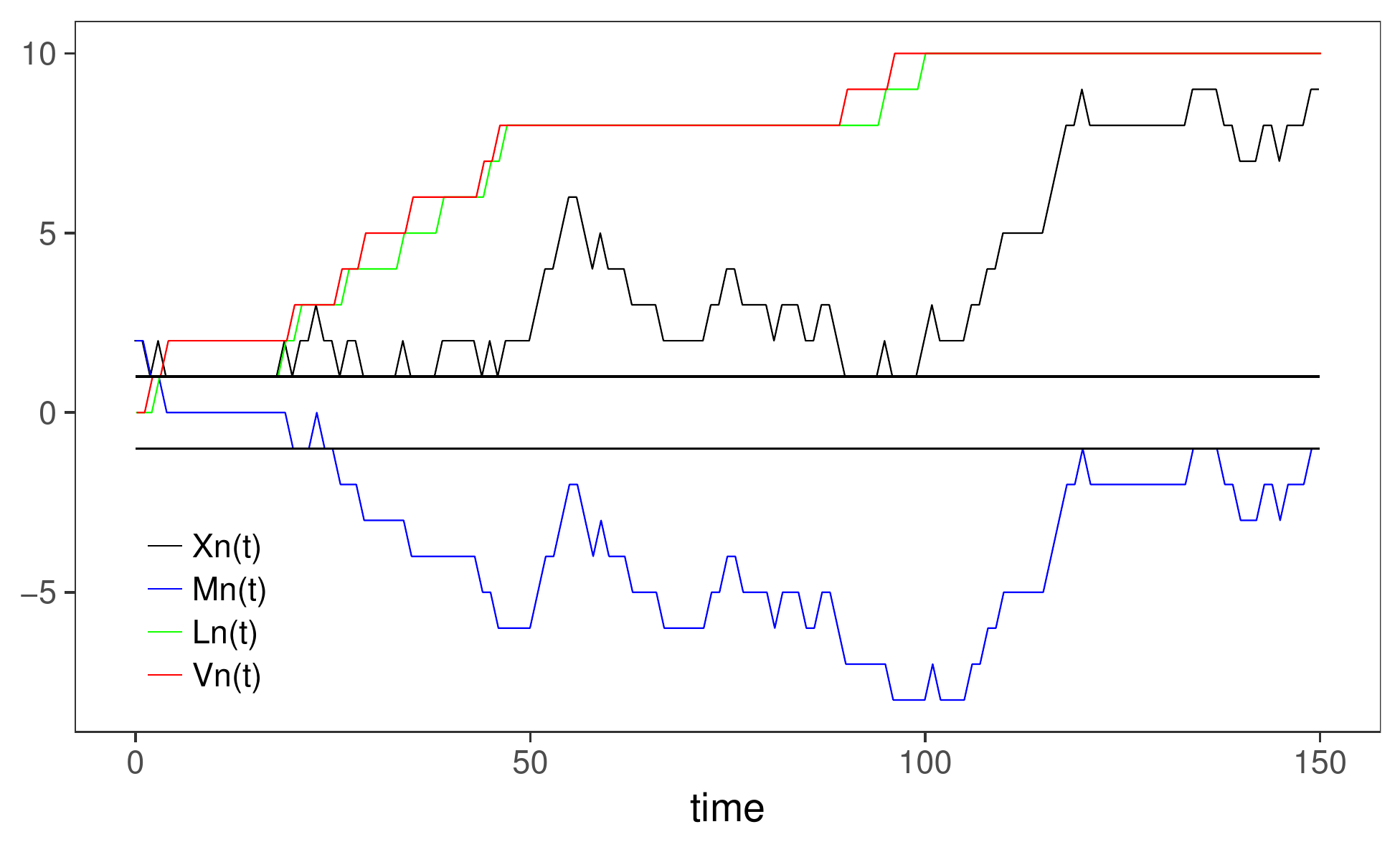}
			\caption{Decomposition of $X_n$} \label{fig:random_walk_decomposition}}
		{\small The black line shows a sample path of $X_n$ for $k=2$, $m=0.4$ and $c_n=0.1$. The blue line is $ M_n $ while the green (resp. red) lines show $ L_n $ (resp. $ V_n $). We see that the excursions of $ M_n $ above its running minimum are given by the excursions of $ X_n $ outside $ \lbrace -1, 1 \rbrace $.}
	\end{figure}
	
	\begin{lemma} \label{lemma:cvg_Mn}
		The process $M_n$ converges in distribution in $( \sko[\R_+]{\R}, d_{sko} )$ to $ M_\infty $, a \Bm with variance parameter $m$ (started from $\abs{X_\infty(0)}$).
	\end{lemma}
	
	We prove this lemma below, but let us first conclude the proof of Lemma~\ref{lemma:absolute_value}.
		
	\begin{proof}[Proof of Lemma~\ref{lemma:absolute_value}]
		Recall that we are already considering a subsequence along which $X_n$ converges to $X_\infty$.
		Passing to the limit in \eqref{martingale_plus_increasing}, we obtain
		\begin{equation*}
		\abs{X_\infty(t)} = M_\infty(t) + \sup_{s\leq t} \left( -M_\infty(s) \right)^+.
		\end{equation*}
		Setting
		\begin{align*}
			L(t) = \sup_{s \leq t} (-M_\infty(s))^+,
		\end{align*}
		we have by Lévy's theorem \cite[Theorem~VI.2.3]{revuz_continuous_2013} that $ (\abs{X_\infty}, L) $ is distributed as $ (\abs{B}, L^B) $, where $ B $ is \Bm (with variance parameter $m$) and $ L^B $ is its local time at zero.
		Hence $ (\abs{X_\infty(t)}, t \geq 0) $ is distributed as reflected \Bm and $ (L(t), t \geq 0) $ is its local time at zero, \textit{i.e.}
		\begin{equation} \label{local_time_Mn}
		L^0_t(X_\infty) = \sup_{s \leq t} (-M_\infty(s))^+.
		\end{equation}
	\end{proof}
		
	To show that $M_n$ converges to \Bm, we note that $M_n$ is a square integrable martingale with predictable variation
	\begin{equation*}
	\langle M_n \rangle_t = m (t - \nu^n(t) )
	\end{equation*}
	where
	\begin{equation} \label{variance_Mn}
	\nu^n(t) = \int_0^t \1{ \abs{X_n(s)} \leq \frac{1}{\sqrt{n}} } ds.
	\end{equation}
	We prove the following lemma in Subsection~\ref{subsec:occupation_time}.
	
	\begin{lemma} \label{lemma:occupation_time}
		For any $t \geq 0$, $\E{\nu^n(t)} = \bigO{\frac{1}{\sqrt{n}}}$.
	\end{lemma}
	
	The proof of Lemma~\ref{lemma:cvg_Mn} is then straightforward.
	
	\begin{proof}[Proof of Lemma~\ref{lemma:cvg_Mn}]
		From \eqref{variance_Mn} and Lemma~\ref{lemma:occupation_time}, $\qvar{M_n} \to mt$ in probability as $n \to \infty$.
		Moreover,
		\begin{equation*}
		\sup_{t\geq 0} \abs{M_n(t) - M_n(t^-)} \leq \frac{1}{\sqrt{n}}
		\end{equation*}
		almost surely.
		Hence, for example from \cite[Proposition~II.1]{rebolledo_central_1980}, $M_n$ converges to \Bm in distribution in $\sko{\R}$ for all $ T>0 $.
	\end{proof}
	
	In passing, we have proved the following lemma.
	\begin{lemma} \label{lemma:convergence_local_time}
		$ \left( X_n, L_n \right) \cvgas[d]{n} \left( X_\infty, L^0(X_\infty) \right) $
	\end{lemma}
	
	\begin{proof}
		From \eqref{definition_Vn} and the convergence of $ M_n $, it is clear that the pair of processes $ (X_n, V_n) $ converges in distribution in $ \sko{\R^2} $ to $ (X_\infty, L_\infty) $, where
		\begin{equation*}
		L_\infty(t) = \sup_{s \leq t} (-M_\infty(s))^+.
		\end{equation*}
		Furthermore, from \eqref{Vn_local_time} and \eqref{local_time}, we have for all $ t\geq 0 $
		\begin{equation*}
		\abs{L_n(t) - V_n(t)} \leq \frac{1}{\sqrt{n}}.
		\end{equation*}
		As a result the pair $ (X_n, L_n) $ converges in distribution in $ \sko{\R^2} $ to $ (X_\infty, L_\infty) $.
		We conclude the proof by noting that $ L_\infty(t) = L^0(X_\infty) $, as shown in \eqref{local_time_Mn}.
	\end{proof}
	
	\subsection{Local time accumulated between crossings} \label{subsec:local_time_rw}
	
	To prove that the local time accumulated by $X_\infty$ at the origin between crossings is a sequence of exponential variables, we show that the number of visits of the random walk $X_n$ to $[-\frac{1}{\sqrt{n}},\frac{1}{\sqrt{n}}]$ before the first time it reaches $-\frac{2}{\sqrt{n}}$ is a geometric \rv.
	
	Let $ (T_i^n, i\geq 0) $ be the sequence of crossing times of $ [-1/\sqrt{n}, 1/\sqrt{n}] $ by $ X_n $, \textit{i.e.} for $n\geq 0$, set $T_0^n = 0$ and
	\begin{equation*}
	T_{i+1}^n = \inf \lbrace t>T_i^n : \sign(X_n(T_i^n)) X_n(t) < - \tfrac{1}{\sqrt{n}} \rbrace.
	\end{equation*}
	Recall also the definition of $ (T_i, i \geq 0) $ in \eqref{def_Ti}.
	
	\begin{lemma} \label{lemma:convergence_crossing_times}
		As $ n $ tends to infinity,
		\begin{equation*}
		\left( X_n, L_n, \proc{T_i^n}{i\geq 0} \right) \cvgas[d]{n} \left( X_\infty, L^0(X_\infty), \proc{T_i}{i\geq 0} \right)
		\end{equation*}
		in $ \sko[{[0,T]^2}]{\R^2} \times \R^\N $.
	\end{lemma}
	
	The proof of Lemma~\ref{lemma:convergence_crossing_times} is given in Subsection~\ref{subsec:convergence_crossing_times}.
		
	\begin{proof}[Proof of Lemma~\ref{lemma:local_time}]
		Let $Y_n$ be the number of visits to $[-\frac{1}{\sqrt{n}},\frac{1}{\sqrt{n}}]$ up to the first crossing time,
		\begin{equation*}
		Y_n = \sum_{i \geq 1} \1{ \sigma_i^n \leq T_1^n }.
		\end{equation*}
		By the Markov property, $Y_n$ is a geometric \rv with parameter
		\begin{equation*}
		p_n = \P[\tfrac{1}{\sqrt{n}}]{ X_n(\sigma_1^n) = - \tfrac{2}{\sqrt{n}} }.
		\end{equation*}
		For $K=2$, $p_n = \frac{c_n}{2(1+c_n)}$ (and in the general case, $p_n \sim \frac{c_n}{K}$ as $n\to \infty$).
		Since $\frac{\sqrt{n}}{2} c_n \to \gamma \in (0,\infty)$,
		\begin{equation*}
		L_n(T_1^n) = \frac{1}{\sqrt{n}} Y_n
		\end{equation*}
		converges in distribution to an exponential \rv with parameter $\gamma$.
		Set $ E^n_i = L_n(T^n_{i+1}) - L_n(T^n_i) $.
		The \rv* $ E_0^n, E_1^n, \ldots $ are \iid by the strong Markov property and by symmetry.
		As a result, $ \proc{E^n_i}{i\geq 0} $ converges in distribution as $ n $ tends to infinity to a sequence $ \proc{E_i}{i\geq 0} $ of \iid exponential \rv* with parameter $ \gamma $.
		By Lemma~\ref{lemma:convergence_crossing_times}, this limit coincides with $ ( L^0_{T_{i+1}}(X_\infty) - L^0_{T_i}(X_\infty) )_{i\geq 0} $ (also note that $ t \mapsto L^0_t(X_\infty) $ is continuous almost surely).
		
		We would like to show that the sequence $ \proc{E_i^n}{i\geq 0} $ is independent of $\tilde{X}_n$, but this fails when $K \geq 2$.
		To circumvent this issue, we tweak $\tilde{X}_n$ so that it ``forgets'' the amount of time $X_n$ spends in $[-\frac{1}{\sqrt{n}},\frac{1}{\sqrt{n}}]$.
		We do this via a time change.
		Set
		\begin{equation*}
		\delta^n_i = \inf \lbrace t > \tau_i^n : X_n(t) \neq X_n(t^-) \rbrace
		\end{equation*}
		and
		\begin{equation*}
		\theta^n(t) = \inf \left\lbrace \theta > 0 : \int_0^\theta \sum_{i \geq 0} \1{ s \notin [\delta_i^n, \sigma_{i+1}^n [ } ds > t \right\rbrace.
		\end{equation*}
		Then $( \tilde{X}_n(\theta^n(t)), t\geq 0 )$ and $( L_n(T_{i+1}^n) - L_n(T_i^n) )_{i\geq 0}$ are independent.
		Furthermore, for $t\geq 0$,
		\begin{equation*}
		\abs{ \int_0^t \sum_{i \geq 0} \1{ s \notin [\delta_i^n, \sigma_{i+1}^n [ } ds - t } \leq \nu^n(t).
		\end{equation*}
		Moreover $ t \mapsto \nu^n(t) $ is nondecreasing, hence, by Lemma~\ref{lemma:occupation_time}, $ (\nu^n(t), t \geq 0) $ converges to 0 uniformly on compact sets in $ L^1 $.
		It follows that $\theta^n(t) \to t$ as $n \to \infty$ uniformly on compact sets in probability.
		As a result, $\tilde{X}_n \circ \theta^n$ converges in the Skorokhod topology to $\abs{X_\infty}$ (also in probability).
		We can thus conclude that $(L^0_{T_{i+1}}(X_\infty) - L^0_{T_i}(X_\infty))_{i\geq 0}$ is independent of $\abs{X_\infty}$.
	\end{proof}
	
	\subsection{Tightness} \label{subsec:tightness}
	
	\begin{proof}[Proof of Lemma~\ref{lemma:tightness}]
		Tightness of the sequence $X_n$ follows from the convergence in distribution of $M_n$ (recall the decomposition \eqref{tilde_decomposition}).
		Reasoning as in \cite{iksanov_functional_2016} (Proof of Lemma~2.1), we show below that for any $\delta > 0$,
		\begin{equation} \label{bound_increments_Xn}
		\sup_{\abs{s-t}<\delta} \abs{X_n(t)-X_n(s)} \leq \tfrac{3}{\sqrt{n}} + 2 \sup_{\abs{s-t}<\delta} \abs{M_n(t)-M_n(s)}.
		\end{equation}
		We can thus write, for $T > 0$ and $\varepsilon > 0$
		\begin{multline*}
		\lim_{\delta \downarrow 0} \limsup_{n \to \infty} \mathbb{P} \Bigg( \sup_{\substack{\abs{t-s}<\delta \\ s, t \in [0,T]}} \abs{X_n(s)-X_n(t)} > \varepsilon \Bigg) \\ \leq \lim_{\delta \downarrow 0} \limsup_{n \to \infty} \mathbb{P} \Bigg( \tfrac{3}{\sqrt{n}} + 2 \sup_{\substack{\abs{t-s}<\delta \\ s, t \in [0,T]}} \abs{M_n(t)-M_n(s)} > \varepsilon \Bigg).
		\end{multline*}
		The \rhs is zero because the sequence $M_n$ converges in distribution in the space $\sko{\R}$, and tightness of $X_n$ in $\sko{\R}$ follows \citep[Theorem~7.3]{billingsley_convergence_1999}.
		Since $ X_n $ is tight in $ \sko{\R} $ for all $ T>0 $, it is tight in $ (\sko[\R_+]{\R}, d) $.
		
		Let us now prove \eqref{bound_increments_Xn}.
		Fix $0 \leq s \leq t$.
		If $\abs{X_n(u)}>\frac{1}{\sqrt{n}}$ for all $u \in [s,t]$, then
		\begin{equation*}
		X_n(t)-X_n(s) = M_n(t)-M_n(s).
		\end{equation*}
		Otherwise, let
		\begin{align*}
		\alpha &= \inf \lbrace u > s : \abs{X_n(u)} \leq \tfrac{1}{\sqrt{n}} \rbrace, \\
		\beta &= \sup \lbrace u < t : \abs{X_n(u)} \leq \tfrac{1}{\sqrt{n}} \rbrace,
		\end{align*}
		and note that
		\begin{align*}
		\abs{X_n(t)-X_n(s)} &\leq \abs{X_n(t)-X_n(\beta)} + \abs{X_n(\beta)-X_n(\alpha)} + \abs{X_n(\alpha)-X_n(s)} \\
		&\leq \tfrac{3}{\sqrt{n}} + \abs{M_n(t)-M_n(\beta)} + \abs{M_n(s)-M_n(\alpha)}.
		\end{align*}
		Inequality \eqref{bound_increments_Xn} thus holds and the proof of Lemma~\ref{lemma:tightness} is complete.
	\end{proof}
	
	\subsection{Occupation time of the barrier} \label{subsec:occupation_time}
	
	\begin{proof}[Proof of Lemma~\ref{lemma:occupation_time}]
		The bound on the expected time spent inside $[-\frac{1}{\sqrt{n}},\tfrac{1}{\sqrt{n}}]$ follows after showing that the expected length of a visit in this set is of order $\frac{1}{n}$ while the expected number of those visits is of order $\sqrt{n}$.
		By the definition of $\nu^n(t)$,
		\begin{align*}
		\nu^n(t) &= \sum_{i\geq 0} \left( \sigma_{i+1}^n \wedge t - \tau_i^n \wedge t \right) \\
		&\leq \sum_{i\geq 0} \left( \sigma_{i+1}^n - \tau_i^n \right) \1{\tau_i^n \leq t}.
		\end{align*}
		By the strong Markov property,
		\begin{equation*}
		\E{\nu^n(t)} \leq \E{ \sum_{i\geq 0} h^n( X_n(\tau_i^n) ) \1{ \tau_i^n \leq t} }
		\end{equation*}
		where $h^n(x) = \E[x]{ \inf \lbrace t>0 : \abs{X_n(t)} > \frac{1}{\sqrt{n}} \rbrace }$.
		By the Markov property, for $i \in \lbrace -1, 0, 1 \rbrace$,
		\begin{equation*}
		n \sum_{j \in E} q_n(i,j) \left( h^n(j/\sqrt{n}) - h^n(i/\sqrt{n}) \right) = -1.
		\end{equation*}
		Also $h^n(x) = 0$ when $\abs{x} > \frac{1}{\sqrt{n}}$.
		Solving these equations for $K = 2$ yields
		\begin{align*}
		h^n \left( \pm \tfrac{1}{\sqrt{n}} \right) = \frac{3}{m n}, && h^n(0) = \frac{3}{m n} + \frac{1}{c_n m n}.
		\end{align*}
		(In the general case, $h^n \left( \floor{\frac{K+1}{2}}/\sqrt{n} \right) = \frac{K+1}{m n}$.)
		For $i \geq 1$, $X_n(\tau_i^n) = \pm \frac{1}{\sqrt{n}}$, hence
		\begin{equation*}
		\E{\nu^n(t)} \leq \frac{1}{c_n m n} + \frac{3}{m n} \E{ \sum_{i\geq 0} \1{ \tau_i^n \leq t } }.
		\end{equation*}
		But the number of visits of $X_n$ to $[-\frac{1}{\sqrt{n}},\tfrac{1}{\sqrt{n}}]$ before time $t$ is less than the number of excursions outside $[-\frac{1}{\sqrt{n}},\tfrac{1}{\sqrt{n}}]$ before the first excursion of length longer than $t$.
		By the Markov property, the latter is a geometric \rv with parameter
		\begin{equation*}
		\P[\tfrac{2}{\sqrt{n}}]{ \tau_0^n > t }.
		\end{equation*}
		But, for $t>0$, there exists $c \in (0, \infty)$ such that \citep[Proposition~4.2.4]{lawler_random_2010}
		\begin{equation*}
		\lim_{n \to \infty} \sqrt{n} \P[\tfrac{2}{\sqrt{n}}]{ \tau_0^n > t } = c.
		\end{equation*}
		Hence, since $\sqrt{n} c_n \to K \gamma \in (0,\infty)$,
		\begin{equation*}
		\E{ \nu^n(t)} \leq \frac{1}{c_n n m} + \frac{3}{m \sqrt{n}} \left( \sqrt{n} \P[\tfrac{2}{\sqrt{n}}]{ \tau_0^n > t } \right)^{-1} = \bigO{\frac{1}{\sqrt{n}}}.
		\end{equation*}
		This concludes the proof of Lemma~\ref{lemma:occupation_time}.
	\end{proof}
	
	\subsection{Convergence of the crossing times} \label{subsec:convergence_crossing_times}
	
	\begin{proof}[Proof of Lemma~\ref{lemma:convergence_crossing_times}]
		From Lemma~\ref{lemma:convergence_local_time}, we already know that
		\begin{equation*}
		( X_n, L_n ) \cvgas[d]{n} ( X_\infty, L^0(X_\infty) ).
		\end{equation*}
		Furthermore, for all $ i \geq 0 $,
		\begin{equation*}
		T_i^n = L_n^{-1}\left( \sum_{k=1}^{i} E_k^n \right),
		\end{equation*}
		where $ t \mapsto L_n^{-1}(t) $ is the right continuous inverse of $ L_n $.
		Since $ \proc{E_i^n}{i\geq 0} $ converges in distribution to $ \proc{E_i}{i\geq 0} $ and $ L_n $ converges in distribution to $ L^0(X_\infty) $, the sequence $ \proc{T_i^n}{n\geq 1} $ is tight in $ \R $ for all $ i \geq 0 $.
		
		As a result the sequence of random variables $ ( X_n, L_n, \proc{T_i^n}{i\geq 0} ) $ is tight in $ \mathrm{D} ([0,T]^2$, $\R^2) \times \R^\N $, where this space is endowed with the product topology.
		Let $ ( X_\infty$, $L^0(X_\infty)$, $(\tilde{T}_i )_{i\geq 0} ) $ be a possible limit point of this subsequence.
		By the Skorokhod embedding theorem, we can assume that there exists (a version of) a subsequence which converges to (a version of) this limit point almost surely.
		For ease of notation we denote this subsequence by $ ( X_n, L_n, \proc{T_i^n}{i\geq 0} ) $.
		
		Let $ \mathcal{N} \subset \Omega $ be the negligible set on which this convergence fails, and suppose that there exists $ \omega \in \Omega \setminus \mathcal{N} $ such that $ \tilde{T}_1(\omega) < T_1(\omega) $.
		We show that for this to happen, one of two very improbable things must occur: either $ \tilde{T}_1(\omega) = \tilde{T}_2(\omega) $ (but remember that $ L_n(T^n_2) - L_n(T^n_1) $ is asymptotically exponentially distributed) or $ X_\infty $ must remain equal to zero for a positive amount of time after $ \tilde{T}_1 $.
		
		Assume \Wlog that $ X_\infty(0) > 0 $ and that $ X_n(0) > 0 $ for all $ n \geq 1 $.
		Then take $ \varepsilon>0 $ such that $ \tilde{T}_1 + \varepsilon < T_1 $ ($ \omega $ is kept fixed in the remainder of the proof).
		Since $ X_n \Rightarrow X_\infty $,
		\begin{equation*}
		\inf \lbrace X_n(s), T_1^n \leq s \leq T_1^n + \varepsilon \rbrace \cvgas{n} \inf \lbrace X_\infty(s), \tilde{T}_1 \leq s \leq \tilde{T}_1 + \varepsilon \rbrace.
		\end{equation*}
		Since $ X_\infty(s) \geq 0 $ for $ s < T_1 $, the \rhs is non-negative while the \lhs is non-positive because $ X_n(T_1^n) = -\frac{2}{\sqrt{n}} $.
		As a result
		\begin{equation*}
		\lim_{n \to \infty} \inf \lbrace X_n(s), T_1^n \leq s \leq T_1^n + \varepsilon \rbrace = 0.
		\end{equation*}
		Also note that
		\begin{equation*}
		\sup \lbrace \abs{X_n(s)}, T_1^n \leq s \leq T_2^n \wedge (T_1^n + \varepsilon) \rbrace \leq \abs{\inf \lbrace X_n(s), T_1^n \leq s \leq T_1^n + \varepsilon \rbrace}.
		\end{equation*}
		Moreover the \lhs converges to
		\begin{equation*}
		\sup \lbrace \abs{X_\infty(s)}, \tilde{T}_1 \leq s \leq \tilde{T}_2 \wedge (\tilde{T}_1 + \varepsilon) \rbrace.
		\end{equation*}
		The latter must then be zero. 
		Hence either $ \tilde{T}_1 = \tilde{T}_2 $ or there exists $ \eta > 0 $ such that $ \abs{X_\infty(s)} = 0 $ for all $ \tilde{T}_1 \leq s \leq \tilde{T}_1 + \eta $.
		Since $ L_n(T_1^n) - L_n(T_2^n) $ converges to an exponential \rv with parameter $ \gamma \in (0,\infty) $ and $ \abs{X_\infty} $ is distributed as reflected \Bm, both these events have probability zero.
		
		Suppose now that $ \tilde{T}_1(\omega) > T_1(\omega) $ for some $ \omega \in \Omega \setminus \mathcal{N} $.
		By the definition of $ T_1 $, there exists $ t \in (T_1, \tilde{T}_1) $ such that $ X_\infty(t) < 0 $.
		Since $ T_1 \to \tilde{T}_1 > t $, there exists $ n_0 $ large enough that $ T_1^n > t $ for all $ n \geq n_0 $.
		Then for all $ n \geq n_0 $, $ X_n(t) \geq -\frac{1}{\sqrt{n}} $, but at the same time $ X_n(t) \to X_\infty(t) < 0 $, leading to a contradiction.
		
		We have thus shown that $ \tilde{T}_1 = T_1 $ almost surely.
		By induction one shows that $ \tilde{T}_i = T_i $ almost surely for all $ i \geq 0 $.
		It follows that $ ( X_\infty, L_\infty, \proc{T_i}{i\geq 0}) $ is the only possible limit point of the sequence $ ( X_n, L_n, \proc{T_i^n}{i\geq 0}) $.
		Together with the tightness of this sequence, this concludes the proof of Lemma~\ref{lemma:convergence_crossing_times}.
	\end{proof}
	
	\bibliography{prBm.bib}
\end{document}